\documentclass[10pt]{amsart}

\usepackage{amsfonts,verbatim,srcltx,amsmath,amsthm,amssymb,amscd,enumerate,eucal,url,stmaryrd}

\usepackage[pagebackref,colorlinks=true]{hyperref}  

\numberwithin{equation}{section}
\usepackage{graphicx}
\newtheorem{thrm}{Theorem}[section]
\newtheorem{lemma}[thrm]{Lemma}
\newtheorem{prop}[thrm]{Proposition}
\newtheorem{cor}[thrm]{Corollary}
\newtheorem{dfn}[thrm]{Definition}

\newtheorem{rmrk}[thrm]{Remark}

\newtheorem{conv}[thrm]{Convention}

\newcommand{\dx}{\partial/\partial x}
\newcommand{\dy}{\partial/\partial y}
\newcommand{\dz}{\partial/\partial z}

\newcommand{\dta}{\partial/\partial t_{a}}
\newcommand{\dxa}{\partial/\partial x_{a}}
\newcommand{\dya}{\partial/\partial y_{a}}
\newcommand{\dza}{\partial/\partial z_{a}}

\newcommand{\f}{\omega}
\newcommand{\C}{\nabla}

\newcommand{\be}{\begin{equation}}
\newcommand{\ee}{\end{equation}}
\newcommand{\ta}{T^a}
\newcommand{\tb}{T^{sym}}

\begin{document}

\begin{abstract}
We introduce the notion of paraquaternionic contact structures (pqc structures), which turns out to be a 
 generalization of the para 3-Sasakian geometry. We derive a distinguished linear connection preserving the pqc structure. Its torsion tensor  is expressed explicitly in terms of the structure tensors and the structure equations of a pqc manifold are presented. We define pqc-Einstein manifolds and show that para 3-Sasakian spaces are precisely pqc manifolds, which are pqc-Einstein. Furthermore, we introduce the paraquaternionic Heisenberg qroup and show that it is the flat model of the pqc geometry.
\end{abstract}

\title[Paraquaternionic contact structures]
{Geometry of paraquaternionic contact structures}
\date{\today}
\author{Marina Tchomakova}
\address[Marina Tchomakova]
{University of Sofia, Faculty of Mathematics and Informatics,
blvd. James Bourchier 5, 1164, Sofia, Bulgaria} \email{marina.chomakova@fmi.uni-sofia.bg}

 \author{Stefan Ivanov}
\address[Stefan Ivanov]{University of Sofia, Faculty of Mathematics and Informatics,
blvd. James Bourchier 5, 1164, Sofia, Bulgaria}
\address{and Institute of Mathematics and Informatics, Bulgarian Academy of
Sciences} \email{ivanovsp@fmi.uni-sofia.bg}




\author{Simeon Zamkovoy}


\maketitle

\setcounter{tocdepth}{2}
\tableofcontents

\section{Introduction}
We investigate and study the sub-Riemannian geometry of 3-contact structures on a $4n+3$-dimensional differentiable manifold related to  the algebra of paraquaternions, known also as split quaternions \cite{Swann}, quaternions of the second kind \cite{L}, complex product structures \cite{AS}. 

In the even, $4n$-dimensional case, the almost paraquaternionic structures are atractive in theoretical physics, string theory due to a closed relation with the Born geometry arising in a natural way in string dinamics (see \cite{FLM,FLM1,FRS} and references there in).

 In the odd, $4n+3$-dimensional case,  the algebra of paraquaternions introduces the notion of para quaternionic contact structures which turns out to be a substantional generalization of the para 3-Sasakian geometry developed in \cite{AK0,Swann}.

Paraquaternionic contact  geometry is a topic with some analogies  with the quaternionic contact  geometry introduced by O.Biquard \cite{B} and its developments in connection with finding  the extremals and the best constant in the $L^2$ Folland-Stein inequality on the quaternionic Heisenberg group and related quaternionic contact Yamabe problem \cite{IMV1,IMV2,IP,IMV3,IV}, but  also with differences mainly because the paraquaternionic contact structure leads  to work with sub-hyperbolic  PDE instead of sub-elliptic PDE in the quaternionic contact case. 

 In this paper, we develope the geometry of paraquaternionic contact structures. 
We define a paraquaternionic contact (pqc) manifold $(M, [g], \mathbb{PQ})$ to be a
$4n+3$-dimensional manifold $M$ with a codimension three
distribution $H$  locally given as the kernel of a 1-form
$\eta=(\eta_1,\eta_2,\eta_3)$ with values in $\mathbb{R}^3$. In addition, $H$ has a conformal  $Sp(n,\mathbb R)Sp(1,\mathbb R)$ structure, i.e. it is
equiped with a conformal class of neutral  metrics $[g]$ of signature $(2n,2n)$ and a rank-three bundle
$\mathbb {PQ}$ consisting of (1,1)-tensors on $H$ locally generated
by two almost para complex structures $I_1,I_2$ and an almost complex structure $I_3$ on $H$ satisfying
the identities of the imaginary unit paraquaternions,
$$I_1^2=I_2^2=id_{H}, \quad I_3^2=-id_H, \quad I_1I_2=-I_2I_1=I_3,$$
such that
$$-2g(I_1X,Y)\ =\ d\eta_1(X,Y),\quad -2g(I_2X,Y)\ =\ d\eta_2(X,Y), \quad 2g(I_3X,Y)\ =\ d\eta_3(X,Y), \quad g\in[g].$$
The 1-form $\eta$ is determined up to a conformal factor and hence
$H$ becomes equipped with a conformal class [g] of neutral
Riemannian metrics of signature (2n,2n). Transformations preserving a
given paraquaternionic contact structure $\eta$, i.e.
$\bar\eta=\mu\Phi\eta$ for a non-vanishing smooth function $\mu$ and an $SO(1,2)$ valued  smooth matrix $\Phi$ are
called \emph{paraquaternionic  contact conformal (pqc conformal)
transformations}.

The main purpose of this paper is to define  a ``canonical'' connection on every pqc-manifold  of dimension at least eleven. We show that in these dimensions there exists a unique space $V$ complementary to $H$, $\quad TM=H\oplus V$, which is locally generated by a three vector fields $\xi_1,\xi_2,\xi_3$ satisfying the relations \eqref{xi} below. For any fixed metric $g\in [g]$, we define the canonical connection as the unique connection preserving the splitting $H\oplus V$ and the $Sp(n,\mathbb R)Sp(1,\mathbb R)$ structure on $H$ with torsion $T$ determined by $T(X,Y)=-[X,Y]_{{_V}}$  and  the endomorphisms $T(\xi,.)_{_H}$ of $H$ lies  in $(sp(n,\mathbb R)+sp(1,\mathbb R))^{\perp}\subset gl(4n)$. We also describe the torsion endomorphism explicitly in terms of the structure tensors  (Theorem~\ref{main1}).

In the seven dimensional case the conditions \eqref{xi} do not hold in general. The existence of such a connection 
requires the pqc structure to satisfy \eqref{xi}. Henceforth, by a pqc structure in dimension 7, we shall mean a pqc-structure satisfying \eqref{xi}. 

We define a global  4-form, express the torsion endomorphism in terms of its exterior derivative and derive  structure equations of a pqc manifold.

We introduce the notion of pqc-Einstein manifold such that the horizontal Ricci tensor of the canonical connection is proportional to the horizontal metric and prove  that the corresponding pqc scalar  curvature (the horizontal trace of the horizontal Ricci tensor) is constant in dimensions bigger than seven. We show that pqc-Einstein condition is equivalent to the vanishing of the torsion endomorphism of the canonical connection.

A basic example of paraquaternionic contact manifold  is provided by a para 3-Sasakian manifold, which can
be defined as a $(4n+3)$-dimensional pseudo Riemannian manifold whose metric cone is a hyper paraK\"ahler (hypersymplectic) manifold \cite{AK0,Swann}.   We characterize (locally) para 3-Sasakian manifolds as a paraquaternionic contact manifold which are pqc-Einstein,  provided the dimension is bigger than seven and the pqc scalar curvature is not zero. (Theorem~\ref{main2}). 

We define the paraquaternionic Heisenberg group and show that any flat pqc manifold is locally isomorphic to the paraquaternionic Heisenberg group.

\begin{conv}\label{conven}
We use the following conventions:
\begin{enumerate}[a)]
\item We shall use $X,Y,Z,U$ to denote horizontal vector fields, i.e. $X,Y,Z,U\in H$;
\item $\{e_1,\dots,e_{n},I_1e_1,\dots,I_1e_{n},Ie_2,\dots,I_2e_{n},I_3e_1,\dots,I_3e_{n}\}$ denotes an adapted
orthonormal basis of the horizontal space $\mathbb H$.;
\item The summation convention over repeated vectors from the basis
$\{e_1,\dots,e_{4n}\}$ will be used,
\begin{multline*}
P(e_b,e_b)=\sum_{b=1}^{4n}g(e_b,e_b)R(e_b,e_b)=\sum_{b=1}^{n}\Big[P(e_b,e_b)-P(I_1e_b,I_1e_b)
-P(I_2e_b,I_2e_b)+P(I_3e_b,I_3e_b)\Big]
\end{multline*}
\item The triple $(i,j,k)$ denotes any cyclic permutation of
$(1,2,3)$. In particular, any equation involving $i,j,k$ holds for
any such permutation.
\item $s$ and $t$ will be any numbers from the set $\{1,2,3\}, \quad
s,t\in\{1,2,3\}$.
\end{enumerate}
\end{conv}

\textbf{Acknowledgements:} We thank Ivan Minchev for many valuable discussins, comments and remarks.

 The research of M. Tch.  is partially  financed by the European Union-Next Generation EU, through the National Recovery and Resilience Plan of the Republic of Bulgaria, project N:
BG-RRP-2.004-0008-C01. The research of S.I.   is partially supported by Contract KP-06-H72-1/05.12.2023 with the National Science Fund of Bulgaria,  by Contract 80-10-181 / 22.4.2024  with the Sofia University "St.Kl.Ohridski" and  the National Science Fund of Bulgaria, National Scientific Program ``VIHREN", Project KP-06-DV-7.


\section{Paraquaternionic contact structure}

\subsection{Paraquaternions}
The algebra $pQ$ of paraquaternions  (sometimes called  split quaternions \cite{Swann})  is a four-dimensional real vector space with
basis ${1, r_1,r_2,r_3}$, satisfying,
\begin{equation}\label{pq1}
r_1^2= r_2^2=1,\quad r_3^2=-1, \quad r_1r_2=-r_2r_1=r_3.
\end{equation}
This carries a natural indefinite inner product given by $<p,q>= Re( \bar pq)$, where
$p = t+r_3x+r_1y+r_2z$ has $\bar p =  t-r_3x-r_1y-r_2z$. Furthermore, we have $||p||^2 = t^2 +x^2 -y^2 -z^2$
to be a metric of signature (2,2). This norm is multiplicative, $||pq||^2 = ||p||^2||q||^2$,
but the presence of elements of length zero means that $pQ$ contains zero divisors.

We introduce the numbers
\begin{equation}\label{eps}
\epsilon_1=\epsilon_2=-\epsilon_3=1, \quad satisfying \quad \epsilon_i\epsilon_j=-\epsilon_k 
\end{equation}
and can therefore write \eqref{pq1} as follows,
\begin{equation}\label{pq2}
r_{i}^2= \epsilon_{i},\quad r_{i}r_{j}=-r_{j}r_{i}=-\epsilon_{k}r_{k}. 
\end{equation}
We recall the definition of the Lie groups $Sp(n,pQ), Sp(1,pQ)$ and $Sp(n,pQ)Sp(1,pQ)$. 
The $4n$ dimensional vector space $pQ^{n}=\mathbb R^{4n}$ inherits the inner product $<P,Q>=Re(\bar P^TQ)$ of signature (2n,2n)
with authomorphism group $Sp(n,pQ)$ isomorphic to $Sp(2n,\mathbb R)$ with Lie algebra $sp(2n,\mathbb R)$. An isomorphism is induced by the correspondence
\[
t+xr_3+yr_1+zr_2\rightarrow \left[ {
\begin{array}{cc}
t+y & x+z\\[2ex]
-x+z & t-y
\end{array}
}\right].
\]

In partucular,  the Lie group
$Sp(1,pQ)\cong Sl(2,\mathbb R)\cong SU(1,1)\cong Sp(1,\mathbb R)$ is the pseudo-sphere in $pQ=\mathbb R^{2,2}$. 

Let $pQ$ act on $pQ^n$ by right multiplications,
$\lambda(p)W=W.p$. This defines a homomorphism $\lambda:\{unit \quad paraquaternions\}\rightarrow SO(2n,2n)$ with the convention that $SO(2n,2n)$ acts on $\mathbb R^{4n}$ on the left. The image is the group $Sp(1,pQ)$. 

Let $\lambda(r_i)=I^o_i$. The Lie algebra of $Sp(1,pQ)$ is $sp(1,pQ)=span\{I^o_3,I^o_1,I^o_2\}$. Therefore, the Lie algebra $sp(1,pQ)=Im(pQ)\cong sp(1,\mathbb R)$. The group $Sp(n,pQ)=\{O\in SO(2n,2n): OB=BO for \quad all \quad B\in Sp(1,pQ)\}$ or $Sp(n,pQ)=\{A\in GL(n,pQ) : \bar A^tA=I\}$, which is a Lie group isomorphic to $Sp(n,\mathbb R)$ and $O\in Sp(n,pQ)$ acts by $(p^1,\dots,p^n)^t\rightarrow O(p^1,\dots,p^n)^t$. The group $Sp(n,pQ)\times Sp(1,pQ)$ acts on $pQ^n$ via $(O,p).b=Ob\bar p$ and this action is isometric with the kernel $\mathbb Z_2=\{\pm(1,1)\}$. Hence, $Sp(n,pQ)Sp(1,pQ)=(Sp(n,pQ)\times Sp(1,pQ))/\mathbb Z_2$ is a subgroup of $SO(2n,2n)$ with a Lie algebra isomorphic to $sp(n,\mathbb R)+sp(1,\mathbb R)\in so(2n,2n)$.

We also recal that the Lie algebra $so(2n,2n)=\{O\in GL(4n): OG+GO^t=0\}$, where $G=(g_{ij})$ is the matrics of a neutral metric, i.e. the matrics $OG$ is skew-symmetric.

\subsection{Paraquaternionic contact manifold} A paraquaternionic contact (pqc) manifold $(M, [g], \mathbb{PQ})$ is a
$4n+3$-dimensional manifold $M$ with a codimension three
distribution $H$,  locally given as the kernel of a 1-form
$\eta=(\eta_1,\eta_2,\eta_3)$ with values in $\mathbb{R}^3$. In addition, $H$ has a conformal  $Sp(n,\mathbb R)Sp(1,\mathbb R)$ structure. The precise definition  follows:
\begin{dfn}
A paraquaternionic contact (pqc) manifold $(M,[g], \mathbb{PQ})$ is a
$4n+3$-dimensional manifold $M$ with a codimension three
distribution $H$, such that
\begin{itemize}
\item[i)] $H$ has a conformal $Sp(n,\mathbb R)Sp(1,\mathbb R)$ structure. That is, it is
equiped with a conformal class of neutral  metrics $[g]$ of signature $(2n,2n)$ and a rank-three bundle
$\mathbb {PQ}$ consisting of (1,1)-tensors on $H$, locally generated
by two almost para complex structures $I_1,I_2$ and an almost complex structure $I_3$ on $H$, satisfying
the identities of the imaginary unit paraquaternions,
\begin{equation}\label{paraq}I_1^2=I_2^2=id_{H},\quad I_3^2=-id_H,\quad I_1I_2=-I_2I_1=I_3;\quad I_s^2=\epsilon_s id_{H},\quad I_iI_j=-I_jI_i=-\epsilon_kI_k,
\end{equation}
which are paraquaternionic compatible with any neutral  metric $g\in [g]$ on $H$,
\begin{equation}\label{param}g(I_1.,I_1.)=g(I_2.,I_2.)=-g(I_3.,I_3.)=-g(.,.); \quad g(I_s.,I_s.)=-\epsilon_sg(.,.), \quad g\in[g].
\end{equation}
\item[ii)] $H$ is locally given as the kernel of a 1-form
$\eta=(\eta_1,\eta_2,\eta_3)$ with values in $\mathbb{R}^3, H=\cap_{s=1}^3 Ker\, \eta_s$ and the following compatibility condition holds 
\be\label{ccon}-2\epsilon_sg(I_sX,Y)\ =\ d\eta_s(X,Y), \quad X,Y\in H.
\ee
\end{itemize}
\end{dfn}
The fundamental 2-forms $\omega_s$ of the para quaternionic structure $\mathbb {PQ}$ are
defined by
\begin{equation}  \label{thirteen}
-2\epsilon_s\omega_s\ =\ \, d\eta_{s|H}.
\end{equation}
If   there is a
globally defined form $\eta$ that anihilates $H$, we denote the
corresponding pqc manifold $(M,\eta)$.

We observe that given a contact form the paraquaternionic structure and the
horizontal metric on $H$ are unique if they exist. We have
\begin{lemma} Let $(M,[g],\mathbb{PQ})$ be a pqc manifold. Then:
\begin{itemize}
\item[a)] If $(\eta,I_s,g)$ and $(\eta',I_s',g')$ are two pqc structures then $I_s=I_s'$ and $g=g'$;
\item[b)] If $(\eta,g)$ and $(\eta',g)$ are two pqc structures on $M$  with $Ker(\eta)=Ker(\eta')=H$ then $\eta'=\Phi\eta$ for some matrix $\Phi\in SO(1,2)$ with smooth functions as entries.
\end{itemize}
\end{lemma}
\begin{proof}
Let the tensors $g,d\eta_3{_|{_H}},d\eta_1{_|{_H}},d\eta_2{_|{_H}},I_3,I_2,I_1$ be given in local coordinates by the matrices $G,R_3,R_1,R_2,J_3,J_1,J_2\in GL(4n)$, respectively. From $I_s^2=\epsilon_s$ and the condition \eqref{ccon} we get
$$
GJ_s=-\frac{\epsilon_s}2R_s, \quad J_k=-\epsilon_kJ_iJ_j=-\epsilon_k\epsilon_iJ_i^{-1}G^{-1}GJ_j=\epsilon_j(GJ_i)^{-1}(GJ_j)=\epsilon_iR_i^{-1}R_j.
$$
The conditions $Ker(\eta)=Ker(\eta')=H$ leads to $\eta'_s=\sum_{t=1}^3\Phi_{st}\eta_t$ for some matrix $\Phi\in GL(3)$. Applying the exterior derivative, we get $d\eta_s'=\sum_{t=1}^3(d\Phi_{st}\wedge\eta_t+\Phi_{st}d\eta_t)$, which if restricted to $H$ gives $g(I_s'X,Y)=\sum_{t=1}^3\Phi_{st}g(I_tX,Y)$. Equivalently, $I_s'=\sum_{t=1}^3\Phi_{st}I_t$. Hence, $\Phi\in SO(1,2)$.
\end{proof}
Besides the non-uniqueness due to the action of $SO(1,2)$, the 1-form $\eta$ 
can be changed by a conformal factor, in the sense that if $\eta$ is a form for
which we can find associated almost para quaternionic structure and metric $g$ as
above, then for any $\Phi\in SO(1,2)$ and a non-vanishing function $\nu$, the form $\eta'=\nu\Phi\eta$  
 has $Ker(\eta')=Ker(\eta)=H$ and determines an associated  paraquaternionic contact structure.

The transformations preserving a given pqc structure $\eta, \eta'=\nu\Phi\eta$ for a nonvanishing  smooth
function $\nu$ and an $SO(1,2)$-matrix $\Phi$ with smooth functions as entries, are called \emph{para quaternionic contact conformal (pqc-conformal) transformations}. The pqc conformal transformations  are   studied  in more details in \cite{CIZ}.

Any endomorphism $\Psi$ of $H$ can be uniquely decomposed with respect to the pqc structure $(pQ,g)$  into four $Sp(n,\mathbb R)$-invariant parts $\Psi=\Psi^{+++}+\Psi^{+--}+\Psi^{-+-}+\Psi^{--+}$, where $\Psi^{+++}$ commutes with all three $I_s$, $\Psi^{+--}$ commutes with $I_1$, and anti-commutes with the others two, etc.
Explicitly, one has
\[
\begin{split}
4\Psi^{+++}=\Psi+I_1\Psi I_1+I_2\Psi I_2-I_3\Psi I_3;\quad
4\Psi^{+--}=\Psi+I_1\Psi I_1-I_2\Psi I_2+I_3\Psi I_3;\\
4\Psi^{-+-}=\Psi-I_1\Psi I_1+I_2\Psi I_2+I_3\Psi I_3;\quad
4\Psi^{--+}=\Psi-I_1\Psi I_1-I_2\Psi I_2-I_3\Psi I_3
\end{split}
\]
The two $Sp(n,\mathbb R)Sp(1,\mathbb R)$-invariant components are given by
\[
\begin{split}
\Psi_{[3]}=\Psi^{+++}=\frac14\Big[\Psi+I_1\Psi I_1+I_2\Psi I_2-I_3\Psi I_3\Big];\\
\Psi_{[-1]}=\Psi^{+--}+\Psi^{-+-}+\Psi^{--+}=\frac14\Big[3\Psi-I_1\Psi I_1-I_2\Psi I_2+I_3\Psi I_3\Big]
\end{split}
\]
with the following characterising conditions
\[
\begin{split}
\Psi=\Psi_{[3]} \Longleftrightarrow 3\Psi-I_1\Psi I_1-I_2\Psi I_2+I_3\Psi I_3=0;\\
\Psi=\Psi_{[-1]} \Longleftrightarrow \Psi+I_1\Psi I_1+I_2\Psi I_2-I_3\Psi I_3=0.
\end{split}
\]
Denoting the corresponding (0,2) tensor via $g$ by the same letter, one
sees that the $Sp(n,\mathbb R)Sp(1,\mathbb R)$-invariant components are the projections on the
eigenspaces of the Casimir operator
\[C=I_1\otimes I_1+I_2\otimes I_2+I_3\otimes I_3
\]
corresponding to the eigenvalues 3 and -1, respectively.  If $n=1$, 
then the space of symmetric endomorphisms commuting with all $I_s$ is 1-
dimensional, i.e. the [3]-component of any symmetric endomorphism $\Psi$ on 
$H$  is proportional to the identity, $\Psi_{[3]}=\frac{Tr(\Psi)}4 Id_{_{|H}}$. Note here that each of the
three 2-forms $\f_s$ belongs to its [-1]-component, $\f_s=\f_{s[-1]}$ and constitutes
a basis of the Lie algebra $sp(1,\mathbb R)$.

Consider the orthogonal complement $(sp(n,\mathbb R)+sp(1,\mathbb R)^{\perp}\subset so(2n,2n)$ of the Lie algebra 
$(sp(n,\mathbb R)+sp(1,\mathbb R)\subset so(2n,2n)$ with respect to the standard neutral inner product $<,>$ on $so(2n,2n)$, comming from the standard neutral inner product of the general linear algebra $gl(4n)$, defined by $<A,B>=Tr(GAGB^t)=<(e_a,e_a)><A(e_a),B(e_a)>, A,B \in gl(4n)$.
It is known that a skew-symmetric
endomorphism $A\in so(2n,2n)$, considered as an element of the orthogonal lie
algebra $so(2n,2n)$, belongs to the orthogonal complement  $(sp(n,\mathbb R)+sp(1,\mathbb R)^{\perp}\subset so(2n,2n)$ if and only if $A$ coincides with the completely trace-free part of its
[-1]-component. More precisely, we have
\be\label{ort}
A\in  (sp(n,\mathbb R)+sp(1,\mathbb R)^{\perp}\subset so(2n,2n) \Longleftrightarrow A=A_{[-1]}-A_{sp(1,\mathbb R)},
\ee
where $A_{sp(1,\mathbb R)}$ is the orthogonal projection of $A$ onto $sp(1,\mathbb R)$ given by  
$4nA_{sp(1,\mathbb R)} = \sum_{s=1}^3 A(e_a,I_se_a)\f_s$.

\section{The canonical connection}
The purpose of this section is to construct our main tool in order to investigate the geometry of  pqc manifolds, namely we construct a canonical connection which preserves the pqc structure having simple torsion. We have
\begin{thrm}\label{main1} {Let $(M, [g],\mathbb{PQ})$ be a pqc
manifold} of dimension $4n+3>7$ with a fixed metric $g\in[g]$. Then, there exists a unique connection
$\nabla$ with
torsion $T$ on $M^{4n+3}$ and a unique supplementary subspace $V$ to $H$ in
$TM$, such that:
\begin{enumerate}
\item[i)]
$\nabla$ preserves the decomposition $H\oplus V$ and the $Sp(n,\mathbb R)Sp(1,\mathbb R)$ structure on $H$,
i.e. 

$\nabla g=0,  \nabla\sigma \in\Gamma(\mathbb PQ)$ for a section
$\sigma\in\Gamma(\mathbb{PQ})$, 
\item[ii)]
for $X,Y\in H$, one has $T(X,Y)=-[X,Y]_{|V}$;
\item[iii)] for $\xi\in V$, the endomorphism $T(\xi,.)_{|H}$ of $H$ lies in
$(sp(n,\mathbb R)\oplus sp(1,\mathbb R))^{\bot}\subset gl(4n)$;
\item[iv)]
the connection on $V$ is induced by the natural identification $\varphi$ of
$V$ with the subspace $sp(1,\mathbb R)$ of the endomorphisms of $H$, i.e.
$\nabla\varphi=0$.
\end{enumerate}
\end{thrm}
In ii), the neutral  inner product $<,>$ of $End(H)$ is given by 

\centerline{$<A,B> =Tr(GAGB^t)= g(e_a,e_a)  g(A(e_a),B(e_a)),$ for $A, B \in End(H)$.}

\begin{proof} The proof of the theorem follows from several propositions and  lemmas and occupies the rest of the section.

Given a pqc manifold $M$, we consider the unique complementary to $H$ in $TM$ subbundle $V, TM=H\oplus V$, which is locally generated by vector fields $\{\xi_1,\xi_2,\xi_3\}$, such that
\begin{equation}  \label{x}
\begin{aligned} \eta_s(\xi_t)=\delta_{st}, \qquad (\xi_s\lrcorner
d\eta_s)_{|H}=0.
\end{aligned}
\end{equation}
where $\lrcorner$ denotes the interior multiplication.
\begin{lemma}\label{l1} Given the splitting $TM=H\oplus V$, there exists a unique $H$-connection $\C$ preserving the horizontal metric $g$ on $H$, $\C g=0$, such that its torsion satisfies $T(X,Y)_{|H}=0$, where the subscript $|_H$ denotes the projection on $H$ in the direction on $V$.
\end{lemma}
The Kozsul formula
$$2g(\C_XY,Z)=Xg(Y,Z)+Yg(X,Z)-Zg(X,Y) +g([X,Y]_{|H},Z)-g([X,Z]_{|H},Y)-g([Y,Z]_{|H},X)
$$
gives the existence and uniquennes of the connection, which proves the assertion.

The condition $T(X,Y)_{|H}=0$ is equivalent to 
\begin{equation}\label{hort}
T(X,Y)=-[X,Y]_{|V}=\sum_{s=1}^3d\eta_s(X,Y)\xi_s=-2\sum_{s=1}^3\epsilon_s\f_s(X,Y)\xi_s.
\end{equation}
The next proposition is our fundamental result.
\begin{prop}\label{fund}
The metric connection $\C$ preserves the paraquaternionic structure on $H$ and   the vertical vector fields $\{\xi_1,\xi_2,\xi_3\}$ satisfy the conditions
\begin{equation}  \label{xi}
\begin{aligned} \eta_s(\xi_t)=\delta_{st}, \qquad (\xi_s\lrcorner
d\eta_s)_{|H}=0.\\ (\xi_j\lrcorner d\eta_i)_{|H}=\epsilon_k(\xi_i\lrcorner
d\eta_j)_{|H}.
\end{aligned}
\end{equation}
\end{prop}
\begin{proof}
To prove this result we involve the contact condition \eqref{ccon}. We use the formula expressing the exterior derivative of a 2-form in terms of a metric connection with torsion and the properties of the covariant derivative of a two form on a paraquaternionic space.

We begin with  the well known formula
\begin{multline}\label{DBC}0=d^2\eta_s(D,B,C)=(\C_Dd\eta_s)(B,C)+(\C_Bd\eta_s)(C,D)+(\C_Cd\eta_s)(D,B)\\+d\eta_s(T(D,B),C) +d\eta_s(T(B,C),D)+d\eta_s(T(C,D),B),\quad D,B,C\in\Gamma(TM).
\end{multline}
In view of \eqref{hort} and \eqref{x} the formula \eqref{DBC} yields to
\begin{multline}\label{no}
\epsilon_iA\Big\{(\C\omega_i)(X,Y,Z)\Big\}=-A\Big\{(\sum_{s=1}^3\epsilon_s\f_s(X,Y)d\eta_i(\xi_s,Z)\Big\}\\= -A\Big{\{\epsilon_j\f_j(X,Y)d\eta_i(\xi_j,Z)+\epsilon_k\f_k(X,Y)d\eta_i(\xi_k,Z)}\Big\},
\end{multline}
where the symbol $A$ stands for a cyclic sum of the arguments $\{XYZ\}$ and $\{ijk\}$ is a cyclic permutation of the numbers $\{123\}$. 

Next, we involve the paraquaternionic multiplications. Since the connection  preserves the horizontal metric, we apply \eqref{paraq} and \eqref{param}  to obtain 
\begin{equation}\label{123}
\begin{aligned}
(\C_X\f_s)(Y,Z)=g((\C_XI_s)Y,Z)=g(\C_XI_sY,Z)+g(\C_XY,I_sZ);\\
(\C_X\f_s)(I_sY,I_sZ)=g(\C_XI^2_sY,I_sZ)+g(\C_XI_sY,I^2_sZ)=\epsilon_s(\C_X\f_s)(Y,Z);\\
(\C_X\f_i)(I_jY,I_jZ)=-\epsilon_kg(\C_XI_kY,I_jZ)-\epsilon_kg(\C_XI_jY,I_kZ)=\epsilon_i(\C_X\f_i)(I_kY,I_kZ).
\end{aligned}
\end{equation}
The  three equalities in \eqref{123} and similar calculations as above,  yields to the equalities
\begin{equation}\label{ijk}
\begin{aligned}
(\C_X\f_i)(I_jY,I_kZ)+(\C_X\f_j)(I_kY,I_iZ)+(\C_X\f_k)(I_iY,I_jZ)=0.
\end{aligned}
\end{equation}
Using the identities \eqref{123} and \eqref{ijk} and after some standard calculations, we derive the following lemma.
\begin{lemma}\label{noo}
The horizontal covariant derivative of the fundamental 2-forms is determined by its skew-symmetric part as follows
\begin{eqnarray}\label{w1}\nonumber
2(\C_X\f_i)(Y,Z) &=& A\Big\{(\C\f_i)(X,Y,Z)\Big\}+\epsilon_iA\Big\{(\C\f_i)(X,I_iY,I_iZ)\Big\}-\epsilon_jA\Big\{(\C\f_j)(I_jX,I_iY,Z)\Big\}\\\nonumber&-&\epsilon_jA\Big\{(\C\f_j)(I_jX,Y,I_iZ)\Big\}
-\epsilon_iA\Big\{(\C\f_k)(I_jX,Y,Z)\Big\}-A\Big\{(\C\f_k)(I_jX,I_iY,I_iZ)\Big\}.\\\nonumber
2(\C_X\f_i)(Y,Z)&=&A\Big\{(\C\f_i)(X,Y,Z)\Big\}+\epsilon_iA\Big\{(\C\f_i)(X,I_iY,I_iZ)\Big\}-\epsilon_kA\Big\{(\C\f_k)(I_kX,I_iY,Z)\Big\}\\\nonumber &-&\epsilon_kA\Big\{(\C\f_k)(I_kX,Y,I_iZ)\Big\}
+\epsilon_iA\Big\{(\C\f_j)(I_kX,Y,Z)\Big\}+A\Big\{(\C\f_j)(I_kX,I_iY,I_iZ)\Big\}.
\end{eqnarray}
\end{lemma}
The two equalities in Lemma~\ref{noo} imply the next identities 
\begin{equation}\label{idxi}
\begin{split}
\epsilon_jA\Big\{(\C\f_j)(I_jX,I_iY,Z)\Big\}+\epsilon_jA\Big\{(\C\f_j)(I_jX,Y,I_iZ)\Big\}
+\epsilon_iA\Big\{(\C\f_k)(I_jX,Y,Z)\Big\}\\+A\Big\{(\C\f_k)(I_jX,I_iY,I_iZ)\Big\}
-\epsilon_kA\Big\{(\C\f_k)(I_kX,I_iY,Z)\Big\}-\epsilon_kA\Big\{(\C\f_k)(I_kX,Y,I_iZ)\Big\}\\
+\epsilon_iA\Big\{(\C\f_j)(I_kX,Y,Z)\Big\}+A\Big\{(\C\f_j)(I_kX,I_iY,I_iZ)\Big\}=0.
\end{split}
\end{equation}

Using \eqref{no}, we obtain from \eqref{idxi} the formulas 
\begin{equation}\label{o1}
\begin{split}
g(X,Y)\Big[d\eta_j(\xi_k,Z)-\epsilon_id\eta_k(\xi_j,Z)\Big]-\f_i(X,Y)\Big[\epsilon_id\eta_j(\xi_k,I_iZ)-d\eta_k(\xi_j,I_iZ)\Big]\\+\f_j(Y,Z)\Big[\epsilon_jd\eta_j(\xi_k,I_jX)+\epsilon_kd\eta_k(\xi_j,I_jX)\Big]+\f_k(Y,Z)\Big[\epsilon_kd\eta_j(\xi_k,I_kX)+\epsilon_jd\eta_k(\xi_j,I_kX)\Big]\\
-g(Z,X)\Big[d\eta_j(\xi_k,Y)-\epsilon_id\eta_k(\xi_j,Y)\Big]-\f_i(Z,X)\Big[\epsilon_id\eta_j(\xi_k,I_iY)-d\eta_k(\xi_j,I_iY)\Big]=0.
\end{split}
\end{equation}
Taking the trace with respect to $X$ and $Y$ into \eqref{o1} we obtain
\begin{equation}\label{xi23}
4(n-1)\Big[d\eta_j(\xi_k,Z)-\epsilon_id\eta_k(\xi_j,Z)\Big]=0,\quad \Longrightarrow \Big[d\eta_j(\xi_k,Z)-\epsilon_id\eta_k(\xi_j,Z)\Big]=0,\quad since \quad n>1.
\end{equation}
Hence, \eqref{xi} holds. 

Next,  we obtain from Lemma~\ref{noo} by applying \eqref{xi} and   \eqref{no} the formulas
\begin{equation}\label{w1fin}
\begin{split}
\C_X\f_i=\alpha_j(X)\f_k+\epsilon_k\alpha_k(X)\f_j, \qquad \C_XI_i=\alpha_j(X)I_k+\epsilon_k\alpha_k(X)I_j,
\end{split}
\end{equation}
where    
\begin{equation}\label{aa1}
\alpha_k(X)=d\eta_i(\xi_j,X)=\epsilon_kd\eta_j(\xi_i,X).
\end{equation}
The proof of Proposition~\ref{fund} is completed.
\end{proof}
Notice that \eqref{xi}  are invariant under the natural $SO(1,2)$-action. 
\subsection{Extension of the partial connection to $V$.}
On the vertical space $V$ we define  the metric of signature (1,2) by 
\[g_{|V}=(\eta_3)^2-(\eta_1)^2-(\eta_2)^2, \quad g(\xi_s,\xi_t)=-\epsilon_s\delta_{st} \] 
to obtain a  metric $g=g_{|H}+g_{|V}$ of signature (2n+1,2n+2) on $M$ by requiring $span\{\xi_1,\xi_2,\xi_3\} = V\perp H$.


Using Proposition~\ref{fund} we extend $\C$ naturally to a H-partial $Sp(n,\mathbb R)Sp(1,\mathbb R)$-connection on
 $V$ as follows
 \begin{lemma}
 The H-partial connection on  $V$ defined by
  \[\C_X\xi=[X,\xi]_{|V}, \xi\in V\]
  is  metric for the metric $g_V$ and it is identified with a connection on the paraquternionic bundle $\mathbb{PQ}=span\{I_1,I_2,I_3\}$ via the identification $\xi_i\rightarrow I_i$.
 \end{lemma}
 \begin{proof}
 The definition of the connection together with \eqref{xi} and \eqref{aa1} yields
 \[\C_X\xi_i=\epsilon_s\sum_{s=1}^3d\eta_s(X,\xi_i)\xi_s=\alpha_j(X)\xi_k + \epsilon_k\alpha_k(X)\xi_j.\]
 Taking into account \eqref{xi} we see that the connection matrix is in $so(1,2)$ and therefore preserves the vertical metric $g_|V$.
 \end{proof}
 
 \subsection{Extension of the adapted partial connection}
 We show in this section how to extend the H-partial connection to a true connection. We recall  the next  general result (see e.g. \cite[Lemma~II.2.1]{B} and its proof).
 \begin{lemma}\label{lbiq}
 Given a complement $V$ to a distribution $H$ with a $\frak K$-structure on $H$ for a group $\frak K\subset GL(n)$ with Lie algebra $\frak k$ there exists a unique $V$-partial $\frak K$-connection on $H$ whose torsion
 \be\label{hortor}
T_{\xi}X=T(\xi_s,X)=\C_{\xi_s}X-[\xi_s,X]_{|_H}, \qquad \xi\in V
\ee
satisfies 
\[T_{\xi}: H\longrightarrow H \in \frak k^{\perp}.\]
 \end{lemma}
 We apply this result to the group $Sp(n,\mathbb R)Sp(1,\mathbb R)$ to obtain the unique $V$-partial $Sp(n,\mathbb R)Sp(1,\mathbb R)$-connection $\C$ on $H$ whose torsion $T_{\xi}\in (sp(n\mathbb R)\oplus sp(1,\mathbb R)^{\perp}$. We can write
\be\label{exh}
\C I_i=\alpha_j I_k+\epsilon_k\alpha_k I_j, \quad \C \f_i=\alpha_j \f_k+\epsilon_k\alpha_k \f_j,
\ee
where the connection 1-forms $\alpha_s(X)$ are given by \eqref{aa1} and $\alpha_s(\xi_t)$  
will be determined explicitly below such that the partial connection  has torsion  satisfying  the conditions of the theorem. 

We extend this partial connection to a true connection  on $M$ defining $\C$ on $V$ by  
\be\label{exv}
\C\xi_i=\alpha_j\xi_k + \epsilon_k\alpha_k\xi_j.
\ee
 It follows from \eqref{exv} that $\C$  preserves   the distribution  $H$ due to the following relation
\[0=Ag(\xi_s,X)=g(\C_A\xi_s,X)+g(\xi_s,\C_AX)=g(\xi_s,\C_AX), \quad A\in\Gamma(TM).
\]

Clearly, $\C$ preserves the extended metric $g=g_{|H}+g_{|V}$ on $M$, the vertical space $V$ and the $Sp(n,\mathbb R)Sp(1,\mathbb R)$ structure on $H$.

It is well known fact that a metric connecion is completely determined by its torsion.
Since  the extended metric is parallel with respect to the extended connection  $\C$, to complete the proof of Theorem~\ref{main1} it is sufficient to  determine the whole torsion $T$ of $\C$  in terms of the data supplied by the paraquaternionic contact structure.  

The difference between the Levi-Civita connection $\C^g$ of the extended metric $g$ and the metric connection $\C$ is given by the well known formula
\be\label{cl}
2g(\C_AB,C)-2g(\C^g_AB,C)=T(A,B,C)-T(B,C,A)+T(C,A,B),
\ee
where $T(A,B,C)=g(T(A,B),C)$ is the torsion of the connection $\C$ and $A,B,C\in\Gamma(TM)$.
\subsection{Determination of the torsion and the vertical connection forms} The torsion on $H$ is given by \eqref{hort}.
We observe that \eqref{exv} agrees with Lemma~\ref{lbiq}. Indeed, 
the conditions \eqref{exv} imply
\be\label{tv}
\begin{split}
d\eta_i(A,B)
=-\epsilon_j(\alpha_j\wedge\eta_k)(A,B)-(\alpha_k\wedge\eta_j)(A,B)-\epsilon_iT(A,B,\xi_i)
\end{split}
\ee
Set  $A=\xi_j, B=X$ into \eqref{tv} to get
\[
d\eta_i(\xi_j,X)=\alpha_k(X)-\epsilon_iT(\xi_i,X,\xi_j)
\]
which, in view of \eqref{aa1}, yields 
$
T(\xi_s,X,\xi_t)=0.
$

We decompose the torsion endomorphism $T(\xi,X,Y)$ into a symmetric  and an anti-symmetric parts,
\begin{equation*}
\begin{split}T(\xi,X,Y)=\tb(\xi,X,Y)+T^a(\xi,X,Y);\\\tb(\xi,X,Y)=\tb(\xi,Y,X), \quad T^a(\xi,X,Y)=-T^a(\xi,Y,X).
\end{split}
\end{equation*}
\begin{prop}
The torsion endomorphism $T(\xi,X,Y)$ is completely trace-free,
\be\label{tr-free}
T(\xi_t,e_a,e_a)=T(\xi_t,I_se_a,e_a)=0.
\ee
\begin{itemize}
\item[a)]
The skew-symmetric part of the torsion endomorphism is given by
\begin{multline}\label{torskew}
\ta(\xi_t,X,Y)=-\frac18\sum_{s=1}^3\epsilon_s\Big[g((\mathbb L_{\xi_t}I_s)X,I_sY)-g((\mathbb L_{\xi_t}I_s)Y,I_sX)-\frac1{4n}g((\mathbb L_{\xi_t}I_s)e_a,e_a)\f_s(X,Y)\Big]
\end{multline}
\item[b)]
The symmetric part of the torsion endomorphism is determined by
\begin{equation}\label{ts}
\tb(\xi_t,X,Y)=\frac12(\mathbb L_{\xi_t}g)(X,Y)
\end{equation}
\end{itemize}
\end{prop}
\begin{proof}
The skew-symmetric $[-1]$-component of \eqref{hortor} is given by
\[
\sum_{s=1}^3\epsilon_sg((\C_{\xi_t}I_s)X,I_sY)-\frac12\sum_{s=1}^3\epsilon_s\{g((\mathbb L_{\xi_t}I_s)X,I_sY)-g((\mathbb L_{\xi_t}I_s)Y,I_sX)\}.
\]
We subtract  its $sp(1,\mathbb R)$-component to obtain \eqref{torskew} since $g((\C_{\xi_t}I_s)X,I_sY)\in sp(1,\mathbb R)$  due to \eqref{exh}. 

Thus, the skew-symmetric part $\ta(\xi,.,.)\in (sp(n,\mathbb R)+sp(1,\mathbb R))^{\perp}\subset so(2n,2n)$
 (c.f. \eqref{ort}). In particular,   $\ta(\xi,.,.)$ 
  satisfies the identities
\be\label{idskew}
\begin{split}
\ta(\xi_t,e_a,e_a)=0,\quad \ta(\xi_t,e_a,I_se_a)=0,\\ \ta(\xi_t,X,Y)-\epsilon_i\ta(\xi_t,I_iX,I_iY)-\epsilon_j\ta(\xi_t,I_jX,I_jY)-\epsilon_k\ta(\xi_t,I_kX,I_kY)=0.
\end{split}
\ee

Since $\C$ preserves the metric and  the splitting $H\oplus V$, we have 
\[
(\mathbb L_{\xi_t}g)(X,Y)
=(\C_{\xi_t}g)(X,Y)+T(\xi_t,X,Y)+T(\xi_t,Y,X)=2\tb(\xi_t,X,Y)
\]
which proves \eqref{ts}. 

In terms of $\C$, the Lie derivative $(\mathbb{L}_{\xi_s}I_t)X$  has the form
\be\label{licar1}
g((\mathbb{L}_{\xi_s}I_t)X,Y)=g((\C_{\xi_s}I_t)X,Y)-T(\xi_s,I_tX,Y)-T(\xi_s,X,I_tY).
\ee
Apply  \eqref{exh} to get  from \eqref{licar1} that
\begin{equation}\label{nnnewts}
g((\mathbb{L}_{\xi_s}I_t)X,X)=g((\C_{\xi_s}I_t)X,X)-2\tb(\xi_s,I_tX,X)=-2\tb(\xi_s,I_tX,X).
\end{equation}
Since the Lee derivative commutes with taking the trace, we obtain from \eqref{ts} and \eqref{nnnewts}
\begin{equation}\label{tr-freeb}
2\tb(\xi_t,e_a,e_a)=(\mathbb L_{\xi_t}g)(e_a,e_a)=0, \quad 
 2\tb(\xi_t,I_se_a,e_a)=-g((\mathbb{L}_{\xi_s}I_t)e_a,e_a)=0
\end{equation}
which combined with the first identity in \eqref{idskew} proves  \eqref{tr-free}.
\end{proof}
The next result finishes the proof of Theorem~\ref{main1}.
\begin{prop}\label{ffin}
The  torsion in the vertical directions  is determined as follows
\be\label{torf}
T(\xi_t,\xi_s,X)=-\epsilon_i(\mathbb L_{\xi_t}\f_i)(\xi_s,I_iX).
\ee
\be\label{torv1}
T(\xi_i,\xi_j,\xi_i)=T(\xi_i,\xi_k,\xi_i)=0. 
\ee
\be\label{torvv}
T(\xi_j,\xi_k,\xi_i)=T(\xi_k,\xi_i,\xi_j)=T(\xi_i,\xi_j,\xi_k)=-\lambda,
\ee
where the function $\lambda$ is given by
\be\label{lamb}
\lambda=\frac1{2n}g((\mathbb L_{\xi_i}I_k)e_a,I_je_a)-\epsilon_id\eta_i(\xi_j,\xi_k)+\epsilon_jd\eta_j(\xi_k,\xi_i) +\epsilon_kd\eta_k(\xi_i,\xi_j)
\ee
The vertical connection 1-forms are determined with the next formulas
\be\label{conv1}
\alpha_j(\xi_i)=-\epsilon_jd\eta_i(\xi_i,\xi_k),\quad \alpha_k(\xi_i)=-d\eta_i(\xi_i,\xi_j),
\ee
\be\label{convv}
\alpha_i(\xi_i)=\frac12\Big[\epsilon_kd\eta_i(\xi_j,\xi_k) +d\eta_j(\xi_k,\xi_i)-\epsilon_id\eta_k(\xi_i,\xi_j)-\epsilon_j\lambda\Big] =\frac{\epsilon_j}{4n}g((\mathbb L_{\xi_i}I_k)I_je_a,e_a).
\ee
\end{prop}
\begin{proof}
We calculate $
(\mathbb L_{\xi_t}\f_i)(\xi_s,I_iX)=-\f_i([\xi_t,\xi_s],I_iX)=-\epsilon_iT(\xi_t,\xi_s,X)$ 
which proves \eqref{torf}.

Set  $A=\xi_i, B=\xi_j, B=\xi_k$  into \eqref{tv} to get
\be\label{vaa}
\begin{split}
d\eta_i(\xi_i,\xi_j)=-\alpha_k(\xi_i)-\epsilon_iT(\xi_i,\xi_j,\xi_i);\quad
d\eta_i(\xi_i,\xi_k)=-\epsilon_j\alpha_j(\xi_i)-\epsilon_iT(\xi_i,\xi_k,\xi_i);\\
d\eta_i(\xi_j,\xi_k)=-\epsilon_j\alpha_j(\xi_j)+\alpha_k(\xi_k)-\epsilon_iT(\xi_j,\xi_k,\xi_i).
\end{split}
\ee
A suitable cyclic sum of the third equality in \eqref{vaa} yields
\begin{multline}\label{vii}
2\alpha_i(\xi_i)=\epsilon_kd\eta_i(\xi_j,\xi_k) +d\eta_j(\xi_k,\xi_i)-\epsilon_id\eta_k(\xi_i,\xi_j)\\+\epsilon_j\Big[ T(\xi_i,\xi_j,\xi_k)- T(\xi_j,\xi_k,\xi_i)+ T(\xi_k,\xi_i,\xi_j)\Big] 
\end{multline}
We will apply the Cartan formula
\be\label{car}
\mathbb{L}_{\xi_k}\f_l=\xi_k\lrcorner(d\f_l)+d(\xi_k\lrcorner\f_l).
\ee
Using \eqref{thirteen} we obtain after some standard calculations that
\begin{equation}\label{om1}
2\f_l=(d\eta_l)_{|H} =-\epsilon_ld\eta_l+\epsilon_l\sum_{s=1}^3\eta_s\wedge(\xi_s\lrcorner d\eta_l)-\epsilon_l\sum_{1\le s\le t\le 3}d\eta_l(\xi_s,\xi_t)\eta_s\wedge\eta_t
\end{equation}
We get from \eqref{car} and \eqref{om1} after some  calculations the next formulas
\begin{gather}\label{liom1}
(\mathbb{L}_{\xi_i}\f_i)(X,Y)=-\epsilon_k\f_j(X,Y)d\eta_i(\xi_i,\xi_j)-\epsilon_j\f_k(X,Y)d\eta_i(\xi_i,\xi_k);\\\label{liom2}
2(\mathbb{L}_{\xi_j}\f_i)(X,Y)=-\epsilon_id(\xi_j\lrcorner d\eta_i)(X,Y)+\epsilon_i(\xi_j\lrcorner d\eta_k)\wedge(\xi_k\lrcorner d\eta_i)(X,Y)\\\nonumber
=-\epsilon_id\alpha_k(X,Y)+\epsilon_j(\alpha_i\wedge\alpha_j)(X,Y)+2d\eta_i(\xi_j,\xi_i)\f_i(X,Y)-2\epsilon_jd\eta_i(\xi_j,\xi_k)\f_k(X,Y);\\\label{liom3}
2(\mathbb{L}_{\xi_i}\f_j)(X,Y)=-\epsilon_jd(\xi_i\lrcorner d\eta_j)(X,Y)+\epsilon_j(\xi_i\lrcorner d\eta_k)\wedge(\xi_k\lrcorner d\eta_j)(X,Y)\\\nonumber
=\epsilon_id\alpha_k(X,Y)-\epsilon_j(\alpha_i\wedge\alpha_j)(X,Y)+2d\eta_j(\xi_i,\xi_j)\f_j(X,Y)-2\epsilon_id\eta_j(\xi_i,\xi_k)\f_k(X,Y),
\end{gather}
where we used \eqref{aa1} and \eqref{xi} to achieve the second identites in \eqref{liom2} and \eqref{liom3}, respectively.

On the other hand, we have  
\begin{equation}\label{licar}
\begin{split}
(\mathbb{L}_{\xi_s}\f_t)(X,Y)
=g((\C_{\xi_s}I_t)X,Y)-T(\xi_s,X,I_tY)+T(\xi_s,Y,I_tX);
\end{split}
\end{equation}
Apply  \eqref{exh} to \eqref{licar}  with $s=t=i$ and the obtained equality compare with \eqref{liom1} to get \eqref{conv1} and 
\begin{gather}\label{torh1}
T(\xi_i,X,I_iY)=T(\xi_i,Y,I_iX) \Longleftrightarrow T(\xi_s,I_sX,I_sY)=\epsilon_sT(\xi_s,Y,X).
\end{gather} 
The first two equalities in \eqref{vaa} together with the alredy proved \eqref{conv1} imply \eqref{torv1}.

The sum  of \eqref{liom2} and \eqref{liom3} yields
\begin{multline*}
(\mathbb{L}_{\xi_i}\f_j)(X,Y)+(\mathbb{L}_{\xi_j}\f_i)(X,Y)=d\eta_i(\xi_j,\xi_i)\f_i(X,Y)+d\eta_j(\xi_i,\xi_j)\f_j(X,Y)\\-[\epsilon_jd\eta_i(\xi_j,\xi_k)+\epsilon_id\eta_j(\xi_i,\xi_k)]\f_k(X,Y)
\end{multline*}
On the othe hand, we obtain from  \eqref{licar} and \eqref{exh} that
\begin{multline*}(\mathbb{L}_{\xi_i}\f_j)(X,Y)+(\mathbb{L}_{\xi_j}\f_i)(X,Y)=\alpha_k(\xi_i)\f_i(X,Y)+\epsilon_k\alpha_k(\xi_j)\f_j(X,Y)\\+[\alpha_j(\xi_j)+\epsilon_i\alpha_i(\xi_i)]\f_k(X,Y)-T(\xi_j,X,I_iY)+T(\xi_j,Y,I_iX)-T(\xi_i,X,I_jY)+T(\xi_i,Y,I_jX)
\end{multline*}
We compare the last two identities and use \eqref{conv1} to get
\begin{multline}\label{th}
[\epsilon_jd\eta_i(\xi_j,\xi_k)+\epsilon_id\eta_j(\xi_i,\xi_k)+\alpha_j(\xi_j)+\epsilon_i\alpha_i(\xi_i)]\f_k(X,Y)\\=
T(\xi_j,X,I_iY)-T(\xi_j,Y,I_iX)+T(\xi_i,X,I_jY)-T(\xi_i,Y,I_jX).
\end{multline}
Taking the trace in \eqref{th} and use  that the torsion endomorphism is completely trace-free, we obtain 
\begin{gather}\label{val}
\epsilon_jd\eta_i(\xi_j,\xi_k)+\epsilon_id\eta_j(\xi_i,\xi_k)+\alpha_j(\xi_j)+\epsilon_i\alpha_i(\xi_i)=0;\\\label{torhh}
T(\xi_j,X,I_iY)-T(\xi_j,Y,I_iX)+T(\xi_i,X,I_jY)-T(\xi_i,Y,I_jX)=0.
\end{gather}

On the other hand, we get from \eqref{vii}
\[\alpha_j(\xi_j)+\epsilon_i\alpha_i(\xi_i)=-\epsilon_jd\eta_i(\xi_j,\xi_k)-\epsilon_id\eta_j(\xi_i,\xi_k)+\epsilon_k[T(\xi_j,\xi_k,\xi_i)-T(\xi_k,\xi_i,\xi_j)],
\]
which compared with \eqref{val} implies \eqref{torvv}.

Substitute \eqref{torvv} into \eqref{vii} to obtain the first equality in \eqref{convv}.


To complete the proof of the Proposition~\ref{ffin} we need to express the function $\lambda$ in terms of the Lie derivatives of the structure. We have  from  \eqref{licar1} using \eqref{exh} that
\begin{multline}\label{n1}
g((\mathbb L_{\xi_i}I_k)I_jX,Y)-g((\mathbb L_{\xi_i}I_k)X,I_jY)=g((\C_{\xi_i}I_k)I_jX,Y))-g((\C_{\xi_i}I_k)X,I_jY))\\
-T(\xi_i,I_kI_jX,Y)-T(\xi_i,I_jX,I_kY)+T(\xi_i,I_kX,I_jY)+T(\xi_i,X,I_kI_jY)\\
=2\epsilon_j\alpha_i(\xi_i)g(X,Y)-\epsilon_iT(\xi_i,I_iX,Y)-T(\xi_i,I_jX,I_kY)+T(\xi_i,I_kX,I_jY)+\epsilon_iT(\xi_i,X,I_iY)
\end{multline}
The trace  in \eqref{n1} gives $$2\epsilon_j\alpha_i(\xi_i)=\frac1{2n}g((\mathbb L_{\xi_i}I_k)I_je_a,e_a).$$  
Combine the latter  with the already proved first identity in \eqref{convv} to get \eqref{lamb} and the second identity in \eqref{convv}. 
\end{proof}
Thus, the  proof of Theorem~\ref{main1} is completed.
\end{proof}
\begin{dfn}
We call the connection, $\C$ constructed in Theorem~\ref{main1}, the canonical paraquaternionic contact connection (canonical pqc-connection). We call the vertical vector fields $\xi_s$  the Reeb vector fields.
\end{dfn}

\subsection{Description of the torsion endomorphism}
Now, we describe the torsion endomorphism. 

The symmetric and anti-symmetric parts of \eqref{torh1} imply
\begin{gather}\label{torhs}
\tb(\xi_s,I_sX,I_sY)=\epsilon_s\tb(\xi_s,X,Y)\Longleftrightarrow \tb(\xi_s,X,I_sY)=\tb(\xi_s,I_sX,Y) \\\label{torha}
\ta(\xi_s,I_sX,I_sY)=-\epsilon_s\ta(\xi_s,X,Y) \Longleftrightarrow \ta(\xi_s,X,I_sY)=-\ta(\xi_s,I_sX,Y)
\end{gather}
The symmetric and skew-symmetric parts of \eqref{torhh}, with the help of \eqref{torh1}, give
\begin{gather}\label{tors}
\tb(\xi_j,I_iX,I_iY)-\epsilon_i\tb(\xi_j,X,Y)+\tb(\xi_i,I_iX,I_jY)+\tb(\xi_i,I_jX,I_iY)=0\\\label{tora}
\ta(\xi_j,I_iX,I_iY)+\epsilon_i\ta(\xi_j,X,Y)+\ta(\xi_i,I_iX,I_jY)+\ta(\xi_i,I_jX,I_iY)=0.
\end{gather}
We  define the tensor $\tau(X,Y)$ on $H$ by the formula
\be\label{deftau}
\tau(X,Y)=-\epsilon_i\tb(\xi_i,I_iX,Y)-\epsilon_j\tb(\xi_j,I_jX,Y)-\epsilon_k\tb(\xi_k,I_kX,Y).
\ee
The tensor $\tau$ does not depend on the particular choice of the Reeb vector fields and is invariant under the natural action of $SO(1,2)$. Indeed, if $\bar\eta_s=\sum_{t=1}^3\Phi_{st}\eta_t, \Phi_{st}\in SO(1,2)$, we have $\bar\xi_s=\sum_{t=1}^3\Phi_{st}\xi_t$ and $\bar I_s=\sum_{t=1}^3\Phi_{st}I_t$, which substituted into \eqref{deftau} does not change it.  The properties of the symmetric part of the torsion endomorphism are encoded in the tensor $\tau$.
\begin{prop}\label{lts}
The $SO(1,2)$-invariant tensor $\tau$ on $H$ is symmetric, trace-free,  belongs to the [-1]- component and determines the symmetric part of the torsion endomorphism, i.e. it satisfies the relations
\begin{gather}\label{tau-trfree}
\tau(X,Y)=\tau(Y,X), \quad \tau(e_a,e_a)=\tau(I_se_a,e_a)=0;\\\label{tau-sym}
\tau(X,Y)-\tau(I_1,X,I_1,Y)-\tau(I_2X,I_2Y)+\tau(I_3X,I_3Y)=0;\\\label{tau-1}
\tb(\xi_s,X,Y)=-\frac14\Big[\tau(I_sX,Y)+\tau(X,I_sY) \Big]
\end{gather}
\end{prop}
\begin{proof}
The formulas in \eqref{tau-trfree} are  consequences  of  \eqref{deftau}, \eqref{torhs} and \eqref{tr-freeb}.

The equality \eqref{tau-sym} follows by a small calculation from  \eqref{deftau} and 
\eqref{torhs}.

To prove \eqref{tau-1} we combine \eqref{deftau}, \eqref{torhs} and \eqref{tors}. It follows  from \eqref{tors} that
\be\label{tors1}
\begin{split}
\tb(\xi_j,I_kX,I_iY)-\epsilon_j\tb(\xi_j,I_jX,Y)+\tb(\xi_i,I_kX,I_jY)+\epsilon_i\tb(\xi_i,I_iX,Y)=0;\\
-\tb(\xi_k,I_jX,I_iY)-\epsilon_k\tb(\xi_k,I_kX,Y)-\tb(\xi_i,I_jX,I_kY)+\epsilon_i\tb(\xi_i,I_iX,Y)=0.
\end{split}
\ee
We calculate from \eqref{deftau} by applying \eqref{tors1} and \eqref{torhs} 
\begin{multline}\label{mid}
\tau(X,Y)+\epsilon_i\tau(I_iX,I_iY)=-2\epsilon_i\tb(\xi_i,I_iX,Y)-\epsilon_j\tb(\xi_j,I_jX,Y)\\+\tb(\xi_j,I_kX,I_iY)-\epsilon_k\tb(\xi_k,I_kX,Y)-\tb(\xi_k,I_jX,I_iY)\\=-4\epsilon_i\tb(\xi_i,I_iX,Y)-\tb(\xi_i,I_kX,I_jY)+\tb(\xi_i,I_jX,I_kY)=-4\epsilon_i\tb(\xi_i,I_iX,Y),
\end{multline}
where we use \eqref{torhs}  in the final step. Clearly, \eqref{mid} is equivalent to \eqref{tau-1}. 
\end{proof}
Next, we characterize the skew-symmetrict part of the torsion endomorphism. We define 
\be\label{defmu}
\mu_s(X,Y)=\epsilon_s\ta(\xi_s,I_sX,Y).
\ee
\begin{prop}\label{lta} The following holds true:
\begin{itemize}
\item[a)] The tensors $\mu_s$ are trace-free, symmetric and equal;
\item[b)] The symmetric trace-free tensor, defined by $\mu=\mu_i$, has the properties
\be\label{propmu}
\mu(I_sX,I_sY)=-\epsilon_s\mu(X,Y)
\ee
and therefore it is $SO(1,2)$-invariant.
\item[c)] If the dimension is seven, then $\mu=0$.
\item[d)] The $SO(1,2)$-invariant tensor $\mu$ determines the skew-symmetric part of the torsion 
 by
\be\label{mus}
\ta(\xi_s,X,Y)=\mu(I_sX,Y).
\ee
\end{itemize}
\end{prop}
\begin{proof}
The equality \eqref{torha} yields
\[\mu_s(X,Y)=\epsilon_s\ta(\xi_s,I_sX,Y)=-\epsilon_s\ta(\xi_s,X,I_sY)=\epsilon_s\ta(\xi_s,I_sY,X)=\mu_s(Y,X)=-\epsilon_s\mu_s(I_sX,I_sY).
\]
Further, \eqref{tora}, \eqref{torha} and \eqref{defmu} together imply
\be\label{muij}
\mu_j(X,Y)-\epsilon_k\mu_j(I_kX,I_kY)+\epsilon_k\mu_i(I_kX,I_kY)-\mu_i(X,Y)=0.
\ee
The equality \eqref{idskew}, written with  the help of \eqref{torha}  in terms of $\mu_s$, reads 
\[
\epsilon_i\mu_i(X,Y)-\mu_i(I_iX,I_iY)-\epsilon_k\mu_i(I_jX,I_jY)-\epsilon_j\mu_i(I_kX,I_kY)=0,
\]
which, in view of \eqref{propmu}, leads  to
\be\label{muijk}
\mu_i(X,Y)=-\epsilon_k\mu_i(I_kX,I_kY)=-\epsilon_j\mu_i(I_jX,I_JY).
\ee
A combination of \eqref{muijk} with \eqref{muij} shows  that $\mu_i=\mu_j=\mu_k=\mu$. 
\end{proof}
If the dimension of M is seven, $n=1$, the conditions \eqref{xi} do not always hold. It follows from the proof of Theorem~\ref{main1} that in dimension seven the canonical pqc-connection exists, if we additionally assume the existence of a complementary to $H$ vertical space $V$, $TM=H\oplus V$, satisfying the properties \eqref{xi} of Lemma~\ref{fund}. In this case,  the  tensor $\mu=0$ and the torsion endomorphism $T(\xi,.,.)$ is symmetric. Henceforth, by a pqc-structure in dimension 7 we shall mean a pqc-structure satisfying \eqref{xi}.


We write Theorem~\ref{main1} in the following more explicit  form
\begin{thrm}\label{main11} {Let $(M, [g],\mathbb{PQ})$ be a pqc
manifold} of dimension $4n+3>7$ with a fixed metric $g\in[g]$. Then, there exists  a unique supplementary subspace $V$ to $H$ in
$TM$ satisfying \eqref{xi} and a unique connection
$\nabla$ with
torsion $T$ on $M$ preserving the splitting $H\oplus V$, the extended metric $g, \C g=0$ and the paraquaternionic structure on $H$, satisfying the  conditions \eqref{exh} and \eqref{exv}, where the connection 1-forms $\alpha_s$ are given by \eqref{aa1}, \eqref{conv1} and \eqref{convv} with $\lambda$ determined by \eqref{lamb}. 

The torsion is determined in \eqref{hort}, \eqref{hortor}, \eqref{torv1}, \eqref{torvv}, \eqref{ts}, \eqref{torskew} with the conditions in Lemma~\ref{lts} and Lemma~\ref{lta}.

Suppose that the Reeb vector fields exist in dimension seven  and denote $V=span\{\xi_1,\xi_2,\xi_3\}$ the vertical space to $H$. Then, all conclusions above are true in dimension seven.
\end{thrm}

Applying \eqref{cl}, we obtain the next corollary.
\begin{cor}
The canonical pqc connection $\C$  and the Levi-Civita connection $\C^g$ of the extended metric $g$ are connected
by
\be\label{clcon}
\begin{split}
g(\C_X\xi_i,Y)=g(\C^g_X\xi_i,Y) +\frac14\Big[\tau(I_iX,Y)+\tau(X,I_iY)\Big]-\f_i(X,Y);\\
g(\C_XY,Z)=g(\C^g_XY,Z);\quad g(\C_{\xi_i}X,Y)=g(\C^g_{\xi_i}X,Y) +\mu(I_iX,Y)-\f_i(X,Y);\\
g(\C_{\xi_i}X,\xi_j)=g(\C^g_{\xi_i}X,\xi_j) +\frac12T(\xi_i,\xi_j,X);\quad 
g(\C_X\xi_i,\xi_j)=g(\C^g_X\xi_i,\xi_j) -\frac12T(\xi_i,\xi_j,X);\\
g(\C_{\xi_k}\xi_i,\xi_j)=g(\C^g_{\xi_k}\xi_i,\xi_j)-\frac12\lambda,\quad  
g(\C_{\xi_i}\xi_i,\xi_j)=g(\C^g_{\xi_i}\xi_i,\xi_j); \quad g(\C_{\xi_j}\xi_i,\xi_j)=g(\C^g_{\xi_j}\xi_i,\xi_j).
\end{split}
\ee
\end{cor}

\section{Basic Examples}
The $4n+4$ dimensional vector space $pQ^{n+1}=\mathbb R^{4n+4}$ has standard coordinates
$$\{t_1,x_1,y_1,z_1\dots,t_{n+1},x_{n+1},y_{n+1},z_{n+1}\}.$$

The standard paraquaternionic  structure
$(J_3,J_1,J_2)$ and the neutral metric $g$ are defined by
\begin{equation*}
\begin{split}
g(\dta,\dta)=g(\dxa,\dxa)=-g(\dya,\dya)=-g(\dza,\dza)=1;\\
J_3\dta=\dxa,\quad J_1\dta=\dya, \quad J_2\dta=\dza \quad for \quad  a=1,\dots,n.
\end{split}
\end{equation*}

\subsection{The paraquaternionic Heisenberg group} Define the 2-step nilpotent group $G(pH) = pQ^n\times Im(pQ)$ with
the group law given by $$(q', \omega') = (q_o, \omega_o)_o(q, \omega) = (q_o + q, \omega_o + \omega + 2 Im(q_o\bar q)),$$
where $q,q_o\in pQ^n$ and $\omega,\omega_o\in Im(pQ)$.

On $G(pH)$ we define the paraquaternionic contact form in paraquaternionic variables as follows
\be\label{pqh}
\tilde\Theta=(\tilde\Theta_3,\tilde\Theta_1\tilde\Theta_2)=\frac12(d\omega-qd\bar q+dq\bar q).
\ee
In real coordinates,  
the structure equations of $G(pH)$ are
\begin{equation*}
\begin{split}
d\tilde\Theta_3=2\Big[dt^a\wedge dx^a+dy^a\wedge dz^a\Big];\\ 
d\tilde\Theta_1=2\Big[dt^a\wedge dy^a+dx^a\wedge dz^a\Big];\\
d\tilde\Theta_2=2\Big[dt^a\wedge dz^a-dx^a\wedge dy^a\Big].\\
\end{split}
\end{equation*}
The left-invariant horizontal  vector fields $T_a, X_a=J_3T_a, Y_a=J_1T_a, Z_a=J_2T_a$ are given by
\begin{equation*}
\begin{split}
T_a=\dta  + 2x^a\dx+2y^a\dy+2z^a\dz;\qquad X_a=\dxa  - 2t^a\dx-2z^a\dy+2y^a\dz;\\
Y_a=\dya  + 2z^a\dx-2t^a\dy-2x^a\dz;\qquad Z_a=\dza  -2y^a\dx+2x^a\dy-2t^a\dz.
\end{split}
\end{equation*}
The horizontal metric of signature (2n,2n) is defined by
$$g(T_a,T_a)=g(X_a,X_a)=-g(Y_a,Y_a)=-g(Z_a,Z_a).$$
The central  (left-invariant vertical) Reeb vector  fields are
$$\xi_3=2\dx, \qquad \xi_1=2\dy, \qquad \xi_2=2\dz.
$$
A small calculation shows the following commutation relations
$$[J_iT_a,T_a]=2\xi_i,\qquad [J_iT_a,J_jT_a]=-2\xi_k.
$$
It is easy to verify that the left-invariant flat connection on $G(pH)$ coincides with the canonical pqc connection of the pqc manifold $(G(pH),\tilde\Theta)$.
This flat pqc structure on the para-quaternionic
Heisenberg group $G(pH)$ turns out to be (locally) the
unique pqc structure with flat canonical connection according to
Theorem~\ref{van} below.

By a hyperbolic rotation of the 1-forms, defining the
horizontal space of $G(pH)$, we obtain an
equivalent pqc-structure (with the same canonical connection). It is
possible to introduce a different not two step nilpotent group
structure on  $ pQ^n\times Im(pQ)$ with
respect to which the rotated forms are left invariant (but not
parallel!). Following is an explicit description of
this construction in dimension seven.

Consider the seven dimensional paraquaternionic Heisenberg group described above.
We  define a
non-left-invariant pqc structure on this manifold as follows. For
each $c\in \mathbb{R}$, let
\begin{equation*} 
\begin{aligned}
& \gamma^1= dt^1, \quad \gamma^4= dz^1, \quad \gamma^7= \tilde\Theta_3\\
&\gamma^2= \sinh(cz^1)\, dx^1 + \cosh(cz^1)\, dy^1,\quad \gamma^3=
\cosh(cz^1)\, dx^1 + \sinh(cz^1)\, dy^1,\\
& \gamma^5= \sinh(cz^1)\, \tilde\Theta_1 + \cosh(cz^1)\, \tilde\Theta_2,\quad \gamma^6=
\cosh(cz^1)\, \tilde\Theta_1 + \sinh(cz^1)\, \tilde\Theta_2.
\end{aligned}
\end{equation*}
A direct calculation shows that for  $c\not=0$ the forms $\{
\gamma^l,\ 1 \leq l\leq 7\}$ define a unique Lie algebra
$\mathfrak {l_0}$ with the following structure equations
\begin{equation}  \label{ex01}
\begin{aligned}
&d\gamma^1=0,\quad d\gamma^2=-c\gamma^{34},\quad d\gamma^3=-c\gamma^{24},\quad d\gamma^4=0,\\
&d\gamma^5=2\gamma^{12}+2\gamma^{34}+c\gamma^{46},\quad
d\gamma^6=2\gamma^{13}+2\gamma^{24}+c\gamma^{45}, \\&
d\gamma^7=2\gamma^{14}-2\gamma^{23}.
\end{aligned}
\end{equation}
In particular, $\mathfrak {l_0}$ is an indecomposable solvable Lie algebra.

Let $e_l, 1 \leq l\leq 7$ be the left invariant vector fields dual
to the 1-forms ${\gamma^l, 1 \leq l\leq 7}$. The (global) flat pqc
structure on $ pQ^n\times Im(pQ)$ can also be
described as follows $\eta_3=\gamma^5, \quad \eta_1=\gamma^6,
\quad \eta_2=\gamma^7, \quad H=span\{\gamma^1,\dots, \gamma^4\}$,
$\omega_3=d\gamma^5_{|_H}=\gamma^{12}+\gamma^{34}, \quad
-\omega_1=d\gamma^6_{|_H}=\gamma^{13}+\gamma^{24}, \quad
-\omega_2=d\gamma^7_{|_H}=\gamma^{14}-\gamma^{23}$. 

It is easy to derive
from \eqref{ex01} that  vector fields $\xi_3=e_5$, $\xi_1=e_6$,
$\xi_2=e_7 $ satisfy the compatibility conditions
\eqref{xi} and therefore the canonical connection exists and
$\xi_s$ are the Reeb vector fields.

Let $(L_0,\eta,pQ)$ be the simply connected Lie
group with Lie algebra $\mathfrak l_0$, equipped with the left
invariant pqc structure $(\eta,pQ)$ defined above. Then, as
a consequence of the above construction, the torsion endomorphism
and the curvature of the canonical connection are identically zero
but the basis $\gamma_1,\dots,\gamma_7$ is not parallel.  The
$Sp(1,\mathbb R)$-connection 1-forms  in the basis $\gamma^1, \dots,
\gamma^7$ are given by $ \alpha_3=0, \quad \alpha_1=0, \quad
\alpha_2=-c\gamma^4. $ 
\subsubsection{An embeding of the paraquaternionic Heisenberg group $G(pH)$}Consider the hypersurface 
$$\Sigma\subset pH^n\times pH:\Sigma=(q',p')\in pH\times pH:Re(p')=-|q'|^2.$$
Clearly, $\Sigma$ is the $0$-level set of $\rho=|q'|^2+t$ and 
\be\label{ppo}d\rho=q'd\bar{q'}+dq'\bar{q'}+dt=2(t^adt^a+x^adx^a-y^ady^a-z^adz^a)+dt.
\ee
Apply the standard para quaternionic structure $J_3,J_1,J_2$ on $\mathbb R^{4n+4}$, induced by the multiplication on the right by the para quaternions $r_3,r_1,r_2\in pH^{n+1}$
to \eqref{ppo}  and compare the result with \eqref{pqh} to get
\begin{equation*}
\begin{split}
J_3d\rho
=2\tilde\Theta_3;\qquad 
J_1d\rho
=2\tilde\Theta_1;\qquad 
J_2d\rho
=2\tilde\Theta_2.
\end{split}
\end{equation*}
We identify $G(pH)$ with $\Sigma$ by $$(q',\omega')\rightarrow (q',p'=-|q'|^2+\omega').$$ Since $dp'=-q'd\bar{q'}-dq'\bar{q'}+d\omega'$, we  write $\tilde\Theta=\frac12(d\omega-q'd\bar{q'}+dq'\bar{q'})=\frac12dp'+dq'\bar{q'}$.

Taking into account that $\tilde\Theta$ is pure imaginary  the last equation takes the form
\be\label{phsp}\tilde\Theta=\frac14(dp'-d\bar{p'})+\frac12(dq'\bar{q'}-q'd\bar{q'}).
\ee

\subsection{Para 3-Sasakian manifolds} We recall the definition of para 3-Sasakian spaces \cite{Swann,ACGL,AK,AK0}.
\begin{dfn}
A $4n+3$-dimensional pseudo-Riemannian manifold $(PS,g)$ with a metric of signature $(2n+1,2n+2)$ is said to be a para 3-Sasakian manifold if it admits three orthogonal Kiling vector fields $\xi_1,\xi_2,\xi_3$ of length squared $g(\xi_s,\xi_s)=-\epsilon_s$ with commutators
\be\label{3sasxi}
[\xi_i,\xi_j]=2\epsilon_k\xi_k
\ee
and the endomorphisms $\Phi_iB=\C^g_B\xi_i$ satisfy
\be\label{phi}(\C^g_A\Phi_i)B=g(\xi_i,B)A-g(A,B)\xi_i.
\ee
\end{dfn}
\noindent
The Kozul formula 
and \eqref{3sasxi} give 
$
\Phi_i\xi_j=-\epsilon_k\xi_k=-\Phi_j\xi_i, \quad \Phi_s\xi_s=0,
$ 
which combined with  \eqref{phi} yields  
\begin{equation*}
\begin{split}
\Phi_s^2A=\epsilon_sA+g(A,\xi_s)\xi_s;\\
\epsilon_k\Phi_k=-\Phi_i\Phi_jA+g(A,\xi_j)\xi_i=\Phi_j\Phi_iA-g(A,\xi_i)\xi_j.
\end{split}
\end{equation*}
It is known that these structures are Einstein with scalar curvature $(4n+3)(4n+2)$  \cite{Swann,ACGL,AK,AK0}.

Consider the 1-forms $\eta_s$ dual to the Killing vector fields $\xi_s$ via the metric, i.e
\[\eta_s(A)=-\epsilon_sg(A,\xi_s).\]
and define $H$ to be  the kernel of a 1-form
$\eta=(\eta_1,\eta_2,\eta_3)$ with values in $\mathbb{R}^3, H=\cap_{s=1}^3 Ker\, \eta_s$. It is easy to see that the restrictions of the metric $g_{_{|H}}$ and of the  endomorphisms $\Phi_s$ on $H$, $I_s=\Phi_s{_{|H}}$ satisfy \eqref{paraq} and \eqref{param} and define a paraquaternionic structure on $H$.

The Killing conditions imply
\be\label{deta}
d\eta_s(A,B)=-2\epsilon_sg(\C^g_A\xi_s,B)=-2\epsilon_sg(\Phi_sA,B)=2\epsilon_sg(A,\Phi_sB).
\ee 
which shows that  compatibility condition $-2\epsilon_sg(I_sX,Y)\ =\ d\eta_s(X,Y), \quad X,Y\in H$ holds. Thus, we have a paraquaternionic contact structure on a para 3-Sasakian spaces and ortogonal splitting $H\oplus\{\xi_i,\xi_2,\xi_3\}$ of the tangent bundle.

The comutators \eqref{3sasxi} yield
\be\label{co}
\begin{split}
d\eta_s(\xi_t,X)=d\eta_s(\xi_t,\xi_s)=0;\qquad d\eta_i(\xi_j,\xi_k)=-2\epsilon_i, \quad d\eta_i=-2\epsilon_i\f_i-2\epsilon_i\eta_j\wedge\eta_k;\\
d\f_i=2\epsilon_j\f_j\wedge\eta_k-2\epsilon_k\f_k\wedge\eta_j.
\end{split}
\ee
Hence, the conditions \eqref{xi} of Theorem~\ref{main1} hold and therefore there exists a canonical connection $\C$ on any para 3-Sasakian space. We have the following proposition.
\begin{prop}\label{p3sas}
The torsion endomorphisms of any para 3-Sasakian structure vanishes, $\tau=\mu=0$.
\end{prop}
\begin{proof}
Since the vector fields $\xi_s$ are Killing, \eqref{ts} shows that the symmetric part of the torsion endomorphism vanish, $0=(\mathbb L_{\xi_s}g)(X,Y)=2\tb(\xi_s,X,Y)$ and therefore $\tau=0$. The general  formula 
\[(\mathbb L_{\xi_s}\f_t)(X,Y)=(\mathbb L_{\xi_s}g)(I_tX,Y)+g((\mathbb L_{\xi_s}I_t)X,Y)
\]
and the Cartan identity \eqref{car} together with \eqref{co} yield
\[g((\mathbb L_{\xi_j}I_i)X,Y)=(\mathbb L_{\xi_j}\f_i)(X,Y)=(\xi_j\lrcorner d\f_i)(X,Y)=-2\epsilon_k\f_k
\]
and the trace-free part of \eqref{n1} implies $\mu=0$.
\end{proof}
An equivalent definition of para-3-Sasakian spaces  is that the cone metric $dt^2+t^2g$ on the product $PS\times \mathbb R^+$ is hypersymplectic (or hyper para Kaehler) \cite{Swann,ACGL,AK,AK0}, i.e. it has holonomy contained in $Sp(n+1,\mathbb R)$. Indeed, the three 2-forms defined by 
\be\label{cone}
F_i=t^2\f_i+t^2\eta_j\wedge\eta_k+\epsilon_it\eta_i\wedge dt
\ee
constitute a hyper para quaternionic structure on the cone and 
\be\label{dfcone}
dF_i=tdt\wedge(2\f_i+2\eta_j\wedge\eta_k+\epsilon_id\eta_i)+t^2d(\f_i+\eta_j\wedge\eta_k)
\ee
Using  \eqref{co} one checks that these forms are closed, $dF_i=0$. After applying the Atiyah-Hitchin  computations from \cite{AH}, one  sees that we have a hypersymplectic (hyper para-Kaehler) structure on the cone.

\subsubsection{The para 3-Sasakian pseudo sphere} An important  explicit example is the  pqc-structure on the para 3-Sasakian pseudosphere. The para 3-Sasakian structure on the pseudosphere (hyperboloid) $pS^{4n+3}=\{|q|^2+|p|^2=1\}\subset pH^n\times pH$ is inherited from the standard flat hypersymplectic structure on $\mathbb R^{4n+4}= pH^n\times pH$.  In paraquaternionic variables,  the  pqc 1-form on the pseudo sphere $pS^{4n+3}=\{|q|^2+|p|^2=1\}\subset pH^n\times pH$ is defined as follows
\begin{equation}\label{p3sas1}
\tilde\eta=dq.\bar q+dp.\bar p-q.d\bar q-p.d\bar p.
\end{equation}
It is shown in \cite{CIZ} that the praquaternionic Heisenberg group $G(pH)$ and the para 3-Sasakian pseudo-sphere $pS^{4n+3}$ are locally paraquaternionic contact conformally equivalent. 

We  explain  briedly the construction from \cite{CIZ}. 
Consider  the map from the pseudo-sphere $pS^{4n+3}$ minus the points $\Sigma_0$,  
$$\Sigma_0=(q,p)\in pS^{4n+3}: |p-1|^2=(t-1)^2+x^2-y^2-z^2=0$$
to the paraquaternionic Heisenberg group $G(pH)\cong\Sigma$, defined by 
\[\mathbb C:\Big(pS^{4n+3}-\Sigma_0\Big)\rightarrow\Sigma,\quad  (q',p')=\mathbb C\Big((q,p)\Big), \quad q'=(p-1)^{-1}q,\quad p'=(p-1)^{-1}(p+1).
\]

The inverse map  $(q,p)=\mathbb C^{-1}\Big((q',p')\Big)$ is given by
\[q=2(p'-1)^{-1}q',\quad  p=(p'-1)^{-1}(p'+1).
\]
It is shown in \cite[Section~3.3]{CIZ} that 
\[2\mathbb C^*\tilde\Theta
=\frac1{|p-1|^2}\lambda.\tilde\eta.\bar{\lambda},\]
where $\lambda=\frac{|p-1|}{p-1}$ is a unit paraquaternion, $\tilde\eta$ is the  paraquaternionic contact form on the pseudo-sphere $pS^{4n+3}$, given by \eqref{p3sas1} and $\tilde\Theta$ is the paraquaternionic contact form on $G(pH)$ written in \eqref{phsp}.

\section{The curvature of the canonical connection}
The main purpose of this section is to show that the curvature of the canonical pqc 
connection is completely determined by its restriction to $H$ and the  torsion endomorphism $\tau$ and $\mu$.

Let $R=[\C,\C]-\C_{[,]}$ be the curvature tensor of $\C$. We denote the curvature of type (0,4) by the same letter, $R(A,B,C,D)=g(R(A,B)C,D), \quad A,B,C,D\in\Gamma(TM)$. 

We define the following Ricci-type contractions:
\begin{gather*}
\rho_s(B,C)=\frac1{4n}R(B,C,e_a,I_se_a)=-\rho_s(C,B), \quad The \quad Ricci\quad 2-forms;\\
Ric(B,C)=R(e_a,B,C,e_a), \quad The \quad pqc-Ricci\quad tensor; \\
Scal=Ric(e_a,e_a), \quad The \quad pqc-scalar\quad curvature;\\
\zeta_s(B,C)=\frac1{4n}R(e_a,B,C,I_se_a); \qquad \varrho_s(B,C)=\frac1{4n}R(e_a,I_se_a,B,C)=-\varrho_s(C,B).
\end{gather*}
The curvature operator $R(B,C)$ preserves the pqc structure on $M$  since the connection $\C$ preserves it. In particular, $R(B,C)$ preserves the splitting $H\oplus V$ and the paraquaternionic structure on $H$, so $R(B,C)\in sp(n,\mathbb R)\oplus sp(1,\mathbb R)$ on $H$. Denote by $R^0$ the $sp(n,\mathbb R)$-part of the curvature, simple calculations yield the decomposition
\be\label{spn1}
R(B,C)X=R^0(B,C)X-\sum_{s=1}^3\epsilon_s\rho_s(B,C)I_sX
\ee
\begin{lemma}
On a pqc manifold the next identities hold
\begin{gather}\label{rjr}
R(B,C)I_iX-I_iR(B,C)X=2\epsilon_i\Big[\rho_j(B,C)I_kX-\rho_k(B,C)I_jX \Big];\\
\label{rhoal}
\rho_i=\frac12\Big[\epsilon_kd\alpha_i-\epsilon_j\alpha_j\wedge\alpha_k \Big];\\\label{rhov}
R(B,C)\xi_i=-2\epsilon_i\rho_k(B,C)\xi_j+2\epsilon_i\rho_j(B,C)\xi_k;
\end{gather}
where the connection 1-forms $\alpha_s$ are given by \eqref{aa1}, \eqref{conv1} and \eqref{convv}.
\end{lemma}
\begin{proof}
The identity \eqref{rjr} follows directly from \eqref{spn1} since $R^0$ commutes with all $I_s$.

Furthermore, the identities \eqref{exh} imply
\begin{multline*}
R(B,C)I_i-I_iR(B,C)=\C_B\C_CI_i-\C_B\C_CI_i-\C_{[B,C]}I_i\\=\C_B[\alpha_j(C)I_k+\epsilon_k\alpha_k(C)I_j]-\C_B[\alpha_j(C)I_k+\epsilon_k\alpha_k(C)I_j]-[\alpha_j([B,C])I_k+\epsilon_k\alpha_k([B,C])I_j\\=(\epsilon_kd\alpha_k+\alpha_i\wedge\alpha_j)(B,C)I_j+(d\alpha_j+\epsilon_j\alpha_k\wedge\alpha_i)(B,C)I_k
\end{multline*}
which compared with \eqref{rjr} proves  \eqref{rhoal}.  

Similarly, using \eqref{exv} we obtain 
$
R(B,C)\xi_i=(\epsilon_kd\alpha_k+\alpha_i\wedge\alpha_j)(B,C)\xi_j+(d\alpha_j+\epsilon_j\alpha_k\wedge\alpha_i)(B,C)\xi_k$ by applying \eqref{rhoal} we get \eqref{rhov}.
\end{proof}
 An application of \eqref{liom2} to the horizontal part of \eqref{rhoal} implies the next
\begin{cor}
The  Ricci 2 forms restricted to $H$ are  given by
\[\rho_i(X,Y)=\epsilon_i(\mathbb{L}_{\xi_k}\f_j)(X,Y)-\epsilon_id\eta_j(\xi_k,\xi_j)\f_j(X,Y)-\epsilon_jd\eta_j(\xi_k,\xi_i)\f_i(X,Y).
\]
\end{cor}
\subsection{The first Bianchi identity and the Ricci-type tensors}
In this section we describe the horisontal Ricci tensors in terms of the torsion
endomorphism of the canonical pqc connection and pqc scalar curvature based on 
Lemma~\ref{lts}, Lemma~\ref{lta}  and the first Bianchi identity.

Let $b(A,B,C)$ be the Bianchi projector,
\begin{equation}  \label{bian01}
b(A,B,C):=\sum_{(A,B,C)}\Bigl\{ (\nabla_AT)(B,C) +
T(T(A,B),C)\Bigr\},
\end{equation}
where $\sum_{(A,B,C)}$ denotes the cyclic sum over the three tangent vectors $A,B,C$.

With this notation the first Bianchi identity reads as
follows
\begin{equation}  \label{bian1}
\sum_{(A,B,C)}\Bigl\{R(A,B,C,D)\Bigr\}= g\Bigl ( b(A,B,C), D\Bigr )=b(A,B,C,D).
\end{equation}

The curvature of a metric connection  is skew-symmetric with respect to
the last two arguments, $R(A,B,C,D)=-R(A,B,D,C)$. It can be derived from the first Bianchi
identity \eqref{bian1} that (see  \cite{B})
\begin{multline}\label{zam}
2R(A,B,C,D)-2R(C,D,A,B)=\ b(A,B,C,D) + \ b(B,C,D,A) \\\ -b(A,C,D,B)-\ b(A,B,D,C).
\end{multline}
\begin{thrm}\label{sixtyseven}
On a $(4n+3)$-dimensional pqc manifold, the horizontal Ricci
tensors $Ric$ and $\zeta_s(X,I_sY)$ are symmetric, the
horizontal Ricci tensors $\rho_s(X,I_sY), \varrho_s(X,I_sY)$ are
symmetric (1,1) tensors with respect to $I_s$,  $$\rho_s(X,I_sY)=-\rho_s(I_sX,Y),\quad \varrho_s(X,I_sY)=-\varrho(I_sX,Y)$$  and the next
formulas hold

\begin{gather}\label{ricci}
Ric(X,Y)  = \frac{Scal}{4n}g(X,Y)+(2n+2)\tau(X,Y)
+(4n+10)\mu(X,Y);\\\label{ricciformf}
 \rho_s(X,I_sY)  = \epsilon_s\frac{Scal}{8n(n+2)}g(X,Y)
+\frac12\Bigl[\epsilon_s\tau(X,Y)-\tau(I_sX,I_sY)\Bigr]+2\epsilon_s\mu(X,Y);\\
\label{riccitau}
 \varrho_s(X,I_sY)   = \epsilon_s\frac{Scal}{8n(n+2)}g(X,Y)
+\frac{n+2}{2n}\Bigl[\epsilon_s\tau(X,Y)-\tau(I_sX,I_sY)\Bigr];
\\\label{riccizeta}
 -\epsilon_s \zeta_s(X,I_sY) \  =\ \frac{Scal}{16n(n+2)}g(X,Y)+
\frac{2n+1}{4n}\tau(X,Y) -\epsilon_s\frac{1}{4n}\tau(I_sX,I_sY)+\frac{2n+1}{2n}\mu(X,Y);
\\\label{scaltor}
Scal\  =\
-8n(n+2)g(T(\xi_1,\xi_2),\xi_3)=8n(n+2)\lambda;\\\label{torv}
 T(\xi_{i},\xi_{j}) =\epsilon_k\frac
{Scal}{8n(n+2)}\xi_{k}-[\xi_{i},\xi_{j}]_{H};
\\\label{vertor}
 T(\xi_i,\xi_j,I_kX)
=\rho_k(I_jX,\xi_i)=-\rho_k(I_iX,\xi_j)=\omega_k([\xi_i,\xi_j],X);
\\\label{ricciformv}
-\epsilon_i\rho_{i}(X,\xi_{i})\  =\ -\frac {X(Scal)}{32n(n+2)} + \frac
12\, \left
(-\rho_{i}(\xi_{j},I_{k}X)+\rho_{j}(\xi_{k},I_{i}X)+\rho_{k}(\xi_{i},I_{j}X)
\right). 
\\\label{ricvert1}
 -\epsilon_i\rho_i(\xi_i,\xi_j)-\epsilon_k\rho_k(\xi_k,\xi_j)=\frac{1}{16n(n+2)}\xi_j(Scal).
\end{gather}
For $n=1$ the above formulas hold with $\mu=0$.
\end{thrm}

\begin{proof}
Since $\nabla$ preserves the orthogonal splitting $H\oplus V$, then \eqref{hort}, \eqref{torv1}and \eqref{torvv} yield
\begin{gather}\label{nablat}
X.T(Y,Z,V)=T(\C_XY,Z,V)=0 \Rightarrow (\C_XT)(Y,Z,V)=0;\\\label{b4}
b(X,Y,Z,V)=\sum_{(X,Y,Z)}T(T(X,Y),Z,V)=-2\sum_{(X,Y,Z)}\sum_{s=1}^3\epsilon_s\f_s(X,Y)T(\xi_s,Z,V)\\\nonumber=-2\sum_{(X,Y,Z)}\sum_{s=1}^3\epsilon_s\f_s(X,Y)\Big[\mu(I_sZ,V)-\frac14\Big( \tau(I_sZ,V)+\tau(Z,I_sV)\Big)\Big]
\end{gather}
We calculate from \eqref{bian1} by taking into account \eqref{nablat}and \eqref{b4} that
\begin{multline*}Ric(X,Y)-Ric(Y,X)=T(T(e_a,X),Y,e_a)+T(T(X,Y),e_a,e_a)+T(T(Y,e_a),X,e_a)\\
=-2\sum_{s=1}^3\epsilon_s\f_s(e_a,X)T(\xi_s,Y,e_a)+2\sum_{s=1}^3\epsilon_s\f_s(e_a,Y)T(\xi_s,X,e_a)=
2\sum_{s=1}^3(T(\xi_s,X,I_sY)-T(\xi_s,Y,I_sX))=0
\end{multline*}
where we used \eqref{torh1} in the final step. Therefore, the horizontal Ricci tensor is symmetric.

The trace in \eqref{rjr} gives
\be\label{riczrho}
Ric(C,I_iY)+4n\zeta_i(C,Y)=-2\epsilon_i\rho_j(C,I_kY)+2\epsilon_i\rho_k(C,I_jY).
\ee
Taking the trace in the first Bianchi identity \eqref{bian1} and using the properties of the curvature,  we obtain
\be\label{b1}
4n\varrho_s(X,Y)+8n\zeta_s(X,Y)=b(e_a,I_se_a,X,Y).
\ee
The trace in \eqref{zam} with an application of \eqref{b1} gives
\begin{multline}\label{b2}
8n\varrho_s(X,Y)-8n\rho_s(X,Y) =b(e_a,I_se_a,X,Y) -b(e_a,I_se_a,Y,X)-2b(e_a,X,Y,I_se_a)\\
=4n\varrho_s(X,Y)+8n\zeta_s(X,Y)-4n\varrho_s(Y,X)-8n\zeta_s(Y,X)-2b(e_a,X,Y,I_se_a).
\end{multline}
We get from \eqref{b1} and \eqref{b2}
\be\label{b3}
\begin{split}
8n\rho_s(X,Y)+ 8n\zeta_s(X,Y)-8n\zeta_s(Y,X)=2b(e_a,X,Y,I_se_a);\\
8n\zeta_s(X,Y)+8n\zeta_s(Y,X)=b(e_a,I_se_a,X,Y)+b(e_a,I_se_a,Y,X);\\
8n\rho_s(X,Y)+ 16n\zeta_s(X,Y)=b(e_a,I_se_a,X,Y)+b(e_a,I_se_a,Y,X)+2b(e_a,X,Y,I_se_a),
\end{split}
\ee
where we used that $\varrho_s$ are skew-symmetric.

We obtain from \eqref{b4} by applying the identity  \eqref{tau-sym} that
\begin{multline}\label{b5}
b(e_a,I_ie_a,Z,V)=T(T(e_a,I_ie_a),Z,V)+2T(T(Z,e_a),I_ie_a,V)\\=8n\mu(I_iZ,V)-2n\tau(I_iZ,V)-2n\tau(Z,I_iV)+4\mu(I_iZ,V)-2\tau(I_iZ,V)+2\tau(Z,I_iV)\\=(8n+4)\mu(I_iZ,V)-(2n+2)\tau(I_iZ,V)-(2n-2)\tau(Z,I_iV),
\end{multline}
\begin{multline}\label{b6}
b(e_a,Z,V,I_ie_a)=T(T(e_a,Z),V,I_ie_a)-T(T(e_a,V),Z,I_ie_a)+T(T(Z,V),e_a,I_ie_a)\\=4\mu(I_iZ,V)+2\tau(I_iZ,V)-2\tau(Z,I_iV).
\end{multline}
Furthermore, since $\nabla$ preserves the orthogonal splitting $H\oplus V$, the first Bianchi identity \eqref{bian1} and \eqref{rhov}
 together with \eqref{hort}, \eqref{torv1} and \eqref{torvv} imply
\begin{multline}\label{pmost}
-2\epsilon_i\rho_i(X,Y)=R(X,Y,\xi_j,\xi_k)=b(X,Y,\xi_j,\xi_k)=
(\nabla_{\xi_j}T)(X,Y,\xi_k)\\+T(T(X,Y),\xi_j,\xi_k)+ T(T(Y,\xi_j),X,\xi_k)+T(T(\xi_j,X),Y,\xi_k)\\
=2(\nabla_{\xi_j}\omega_k)(X,Y)-2T(X,Y,\nabla_{\xi_j}\xi_k)\\ -2\epsilon_i\omega_i(X,Y)T(\xi_i,\xi_j,\xi_k)+
2\omega_k(T(\xi_j,X),Y)-2\omega_k(T(\xi_j,Y),X)\\=2\epsilon_i\lambda\omega_i(X,Y)-
2T(\xi_j,X,I_kY)+2T(\xi_j,Y,I_kX),
\end{multline}
where we used  that $T(\xi_s,X)$ is a horizontal vector field to
conclude the vanishing of terms of the type $(\nabla_AT)(X,\xi_j,\xi_k)$ and \eqref{hort}, \eqref{exh} and \eqref{exv}
to see that $(\nabla_{\xi_j}\omega_k)(X,Y)-T(X,Y,\nabla_{\xi_j}\xi_k)=0$ .

Applying \eqref{mus}, \eqref{tau-1}, Lemma~\ref{lts} and Lemma~\ref{lta}, we calculate from \eqref{pmost} that
\begin{multline}\label{rrho1}
-\epsilon_i\rho_i(X,Y)=\epsilon_i\lambda\f_i(X,Y)\\-\tb(\xi_j,X,I_kY)-\ta(\xi_j,X,I_kY)+\tb(\xi_j,Y,I_kX)+\ta(\xi_j,Y,I_kX)\\=
\mu(I_jY,I_kX)-\mu(I_jX,I_kY)+\frac14\Big[\tau(I_jX,I_kY)+\tau(X,I_jI_kY)-\tau(I_jY,I_kX)-\tau(Y,I_jI_kX) \Big]\\=
\epsilon_i\lambda\f_i(X,Y)+2\epsilon_i\mu(I_iX,Y)+\epsilon_i\frac14\Big[\tau(I_iX,Y)-\tau(X,I_iY) \Big]+\frac14\Big[\epsilon_k\tau(I_jX,I_jI_iY)+\epsilon_j\tau(I_kX,I_kI_iY) \Big]\\
=\epsilon_i\lambda\f_i(X,Y)+2\epsilon_i\mu(I_iX,Y)+\epsilon_i\frac12\Big[\tau(I_iX,Y)-\tau(X,I_iY) \Big]
\end{multline}
We obtain from \eqref{b2}, \eqref{rrho1}, \eqref{b5} and \eqref{b6}
\begin{multline}\label{b7}
8n\varrho_i(X,Y)=8n\rho_i(X,Y) +b(e_a,I_ie_a,X,Y) -b(e_a,I_ie_a,Y,X)-2b(e_a,X,Y,I_ie_a)\\
=-8n\lambda\f_i(X,Y)-16\mu(I_iX,Y)-4n\Big[\tau(I_iX,Y)-\tau(X,I_iY) \Big]\\
+(16n+8)\mu(I_iX,Y)-4\tau(I_iX,Y)+4\tau(X,I_iY)- 8\mu(I_iX,Y)-4\tau(I_iX,Y)+4\tau(X,I_iY)\\
=-8n\lambda\f_i(X,Y)-(4n+8)\Big[\tau(I_iX,Y)-\tau(X,I_iY) \Big]
\end{multline}
We get from \eqref{b3},\eqref{b5}, \eqref{b6} and \eqref{rrho1}

\begin{multline}\label{zet}
16n\zeta_i(X,Y)=-8n\Big[-\lambda\f_i(X,Y)-2\mu(I_iX,Y)-\frac12\Big[\tau(I_iX,Y)-\tau(X,I_iY ) \Big]\\+
(8n+4)\mu(I_iX,Y)-(2n+2)\tau(I_iX,Y)-(2n-2)\tau(X,I_iY)\\+
(8n+4)\mu(I_iY,X)-(2n+2)\tau(I_iY,X)-(2n-2)\tau(Y,I_iX)\\+
 8\mu(I_iX,Y)+4\tau(I_iX,Y)-4\tau(X,I_iY)\\
 =8n\lambda\f_i(X,Y)+(16n+8)\mu(I_iX,Y)+4\tau(I_iX,Y)-(8n+4)\tau(X,I_iY)
\end{multline}
Using \eqref{zet} and \eqref{rrho1}, we obtain from \eqref{riczrho}
\begin{multline}\label{riczrho1}
Ric(X,I_iY)=-4n\zeta_i(X,Y)-4\lambda\f_i(X,Y)-8\mu(I_iX,Y)+\tau(I_iX,Y)+\tau(X,I_iY)\\=
-2n\lambda\f_i(X,Y)-(4n+2)\mu(I_iX,Y)-\tau(I_iX,Y)+(2n+1)\tau(X,I_iY)\\
-4\lambda\f_i(X,Y)-8\mu(I_iX,Y)+\tau(I_iX,Y)+\tau(X,I_iY)\\=-(2n+4)\lambda\f_i(X,Y)+(2n+2)\tau(X,I_iY)+(4n+10)\mu(X,I_iY)
\end{multline}
Since the tensors $\tau$ and $\mu$ are completely trace-free, the trace in \eqref{riczrho1} yields
\be\label{scal}
\lambda=\frac{Scal}{8n(n+2)}.
\ee

By substituting \eqref{scal} into \eqref{riczrho1}, \eqref{rrho1}, \eqref{b7}, \eqref{zet} we get the proof of  \eqref{ricci}-\eqref{torv}.

Furthermore, \eqref{rhov} , the first Bianchi adentity \eqref{bian1} and the fact that $\C$ preserves the ortoghonal splitting $H\oplus V$ imply
\begin{multline}\label{vt}
-\epsilon_k2\rho_k(\xi_j,X)=R(\xi_j,X,\xi_i,\xi_j)=\sum_{\xi_i,\xi_j,X}\{(\C_{
\xi_i}T)(\xi_j,X,\xi_j) +T(T(\xi_i,\xi_j),X,\xi_j) \}\\=(\C_{
X}T)(\xi_i,\xi_j,\xi_j) +T(T(\xi_i,\xi_j),X,\xi_j)=-T(\xi_i,\C_X\xi_j,\xi_j)-T(\xi_i,\xi_j,\C_X\xi_j)-T([\xi_i,\xi_j]_{|_{H}},X,\xi_j)\\=
-2\f_j([\xi_i,\xi_j],X)=-2T(\xi_i,\xi_j,I_jX),
\end{multline}
where we used \eqref{hort}, \eqref{exv}, \eqref{hortor} and the just proved \eqref{torv}. So, \eqref{vertor} follows from \eqref{vt} because 
\[
-\epsilon_k2\rho_k(\xi_i,X)=R(\xi_i,X,\xi_i,\xi_j)=-R(\xi_i,X,\xi_j,\xi_i)=2T(\xi_j,\xi_i,I_iX)=-2T(\xi_i,\xi_j,I_iX).
\]
Similarly, by taking into account \eqref{hort} and \eqref{vertor} we have
\begin{multline}\label{vvrho}
2(-\epsilon_i\rho_i(X,\xi_i)-\epsilon_j\rho_j(X,\xi_j))=R(X,\xi_i,\xi_j,\xi_k)+R(X,\xi_j,\xi_k,\xi_i)\\
=\sum_{\xi_i,\xi_j,X}\{(\C_{\xi_i}T)(\xi_j,X,\xi_k) +T(T(\xi_i,\xi_j),X,\xi_k) \}=(\C_{
X}T)(\xi_i,\xi_j,\xi_k) +T(T(\xi_i,\xi_j),X,\xi_k)\\
=-\frac{X(Scal)}{8n(n+2)}-2\f_k([\xi_i,\xi_j],X)=-\frac{X(Scal)}{8n(n+2)}-2\rho_k(I_jX,\xi_i)
\end{multline}
Making a cyclic permutation of $\{i,j,k\}$ into \eqref{vvrho}, summing the first and the third and subtracting the second, we obtain \eqref{ricciformv}.

We apply \eqref{torv} and calculate
\begin{multline*}
-2(\epsilon_i\rho_i(\xi_i,\xi_j)+\epsilon_k\rho_k(\xi_k,\xi_j))=R(\xi_i,\xi_j,\xi_j,\xi_k)+R(\xi_k,\xi_j,\xi_i,\xi_j)\\=
-\sum_{\xi_i,\xi_j,\xi_k}\{(\C_{\xi_i}T)(\xi_j,\xi_k,\xi_j) +T(T(\xi_i,\xi_j),\xi_k,\xi_j) \}=\frac{\xi_j(Scal)}{8n(n+2)}
\end{multline*}

Finally, \eqref{ricvert1} follows. The proof is complete.
\end{proof}
Due to \eqref{scal} we call the function $\lambda$ \emph{the normalized pqc scalar curvatur} which also satisfies \eqref{lamb}.

Based on \eqref{scaltor}, \eqref{lamb} and Theorem~\eqref{sixtyseven}, we get the next corollary.
\begin{cor}
The pqc scalar curvature $Scal$ does not depend on the canonical pqc connection. It is given by
$
Scal =8n(n+2)\Big[-\frac1{2n}g((\mathbb L_{\xi_j}I_i)I_ke_a,e_a)-\epsilon_jd\eta_j(\xi_k,\xi_i) +\epsilon_kd\eta_k(\xi_i,\xi_j)+\epsilon_id\eta_i(\xi_j,\xi_k) \Big]
$
and satisfies the equalities
$
Scal=2(n+2)\rho_s(I_se_a,e_a)=2(n+2)\varrho_s(I_se_a,e_a)=-4(n+2)\zeta_s(I_se_a,e_a).
$
\end{cor}
Comparing the $Sp(n,\mathbb R).Sp(1,\mathbb R)$-parts  of the Ricci-type tensors from Theorem~\ref{sixtyseven} we conclude the following corollary.

\begin{cor}
The tensor $\tau$ determines the traceless [-1]-component of the horizontal
Ricci-type tensors while the tensor $\mu$ determines the traceless part of the
[3]-component of the horizontal Ricci-type tensors. For example, \eqref{ricci} implies
$
\tau=\frac1{2n+2}Ric_{[-1]},\qquad \mu=\frac1{4n+10}Ric_{[3][0]}.
$
\end{cor}

\section{The second Bianchi identity and the curvature of the  pqc connection}
In this section we describe the curvature of $\C$ and show that the whole curvature is determined from the horizontal curvature. 
We have the following theorem.
\begin{thrm}
\label{bianrrr} On a pqc manifold the curvature of the canonical connection satisfies the
equalities:
\begin{multline}\label{vert1}
 R(\xi_i,X,Y,Z)= -(\nabla_X\mu)(I_iY,Z)
 -\frac14\Big[(\nabla_Y\tau)(I_iZ,X)+(\nabla_Y\tau)(Z,I_iX)\Big]\\ +\frac14\Big[(\nabla_Z\tau)(I_iY,X)+
 (\nabla_Z\tau)(Y,I_iX)\Big]
 +\omega_j(X,Y)\rho_k(I_iZ,\xi_i)-\omega_k(X,Y)\rho_j(I_iZ,\xi_i)\\
-\omega_j(X,Z)\rho_k(I_iY,\xi_i)+\omega_k(X,Z)\rho_j(I_iY,\xi_i)
-\omega_j(Y,Z)\rho_k(I_iX,\xi_i)+\omega_k(Y,Z)\rho_j(I_iX,\xi_i).
\end{multline}
\begin{multline}\label{vert2}
 R(\xi_i,\xi_j,X,Y)=(\nabla_{\xi_i}\mu)(I_jX,Y)-(\nabla_{\xi_j}\mu)(I_iX,Y) +\epsilon_j(\nabla_X\rho_k)(I_iY,\xi_i)\\
 -\frac14\Big[(\nabla_{\xi_i}\tau)(I_jX,Y)+(\nabla_{\xi_i}\tau)(X,I_jY)\Big]
 +\frac14\Big[(\nabla_{\xi_j}\tau)(I_iX,Y)+(\nabla_{\xi_j}\tau)(X,I_iY)\Big]\\ +\epsilon_k\frac{Scal}{8n(n+2)}T(\xi_k,X,Y)
 -T(\xi_j,X,e_a)T(\xi_i,e_a,Y)+T(\xi_j,e_a,Y)T(\xi_i,X,e_a),
\end{multline}
where the Ricci 2-forms are given by
\begin{multline}\label{vert023}
3(2n+1)\rho_i(\xi_i,X)=-\epsilon_i\frac14(\nabla_{e_a}\tau)(e_a,X)-\frac34(\nabla_{e_a}\tau)(I_ie_a,I_iX)\\
+\epsilon_i(\nabla_{e_a}\mu)(X,e_a)
-\epsilon_i\frac{2n+1}{16n(n+2)}X(Scal),
\end{multline}
\begin{multline}\label{vert024}
3(2n+1)\rho_i(I_kX,\xi_j)=-3(2n+1)\rho_i(I_jX,\xi_k)=-\frac{(2n+1)(2n-1)}{16n(n+2)}X(Scal)\\
+2(n+1)(\nabla_{e_a}\mu)(X,e_a)
+\frac{4n+1}4(\nabla_{e_a}\tau)(e_a,X)-\epsilon_i\frac34(\nabla_{e_a}\tau)(I_ie_a,I_iX).
\end{multline}
\end{thrm}

\begin{proof}

We know $R(X,Y,Z,\xi)=0$ since $\C$ preserves the splitting $H\oplus V$. Therefore, \eqref{zam} yields
\be  \label{zamv}
2R(\xi_i,X,Y,Z)= b(\xi_i,X,Y,Z) +  b(X,Y,Z,\xi_i)-  b(\xi_i,Y,Z,X)- b(\xi_i,X,Z,Y).
\ee

We calculate using \eqref{bian01},  applying \eqref{hort} and Theorem~\ref{sixtyseven} that $b(X,Y,Z,\xi_i)=0$ and 

\begin{multline}  \label{bver1}
b(\xi_i,X,Y,Z)=-(\nabla_XT)(\xi_i,Y,Z)+(\nabla_YT)(\xi_i,X,Z) \\
+2\omega_j(X,Y)\rho_k(I_iZ,\xi_i)- 2\omega_k(X,Y)\rho_j(I_iZ,\xi_i)\\=
\frac14(\nabla_X\tau)(I_iY,Z)+\frac14(\nabla_X\tau)(Y,I_iZ)-(\nabla_X\mu)(I_iY,Z)\\
-\frac14(\nabla_Y\tau)(I_iX,Z)-\frac14(\nabla_Y\tau)(X,I_iZ)+(\nabla_Y\mu)(I_iX,Z)\\
+2\omega_j(X,Y)\rho_k(I_iZ,\xi_i)- 2\omega_k(X,Y)\rho_j(I_iZ,\xi_i),
\end{multline}
where we used Lemma~\ref{lts}, Lemma~\ref{lta} and the  equalities  \eqref{exh} and \eqref{exv}
to pass from the second to the third equality. Substitute \eqref{bver1} into \eqref{zamv} to get \eqref{vert1}.

The first Bianchi identity \eqref{bian1} and the fact that $\nabla$ preserves the splitting $H\oplus V$ imply

\begin{multline}\label{nnn1}
R(\xi_i,\xi_j,X,Y)=(\nabla_{\xi_i}T)(\xi_j,X,Y)-(\nabla_{\xi_j}T)(\xi_i,X,Y)+(\nabla_XT)(\xi_i,\xi_j,Y)\\
+\epsilon_k\frac{Scal}{8n(n+2)}T(\xi_k,X,Y)+T(T(\xi_j,X),\xi_i,Y)-T(T(\xi_i,X),\xi_j,Y),
\end{multline}
where we applied \eqref{torv}. Evaluating the first two terms similarly as in \eqref{bver1} and  the third term using  equality \eqref{vertor}, we obtain \eqref{vert2} from \eqref{nnn1}.

The trace in \eqref{vert1} leads to the next equality
\begin{multline}\label{vert011}
n\rho_i(\xi_i,X)-\frac12\epsilon_i\rho_k(I_jX,\xi_i)-
\frac12\epsilon_i\rho_j(I_iX,\xi_k)=-\frac18\epsilon_i(\nabla_{e_a}\tau)(e_a,X)-\frac18(\nabla_{e_a}\tau)(I_ie_a,I_iX).
\end{multline}

The sum of \eqref{vert011} and \eqref{ricciformv} gives
\begin{multline}\label{vert012}
(n+1)\rho_i(\xi_i,X)-\frac12\epsilon_i\rho_i(I_kX,\xi_j)=-\frac18\epsilon_i(\nabla_{e_a}\tau)(e_a,X)-\frac18(\nabla_{e_a}\tau)(I_ie_a,I_iX)
-\frac{\epsilon_iX(Scal)}{32n(n+2)}.
\end{multline}

We involve the second Bianchi identity
\begin{equation}  \label{secb}
\sum_{(A,B,C)}\Big\{(\nabla_AR)(B,C,D,E)+R(T(A,B),C,D,E)\Big\}=0
\end{equation}

which combined with  \eqref{hort} implies 
\begin{equation}  \label{bi20}
\sum_{(X,Y,Z)} \Bigl[(\nabla_XR)(Y,Z,V,W)-2\sum_{s=1}^3\epsilon_s\omega_s(X,Y)R(%
\xi_s,Z,V,W)\Bigr]=0.
\end{equation}
The trace in \eqref{bi20} leads to
\begin{multline}  \label{bi21}
(\nabla_{e_a}R)(I_ie_a,Z,V,W)+2n(\nabla_Z\varrho_i)(V,W)
+2(2n-1)R(\xi_i,Z,V,W)\\-2\epsilon_iR(\xi_j,I_kZ,V,W)+2\epsilon_iR(\xi_k,I_jZ,V,W)=0.
\end{multline}
After taking the trace in \eqref{bi21} and applying the formulas in Theorem~\ref{sixtyseven}
we come to
\begin{multline}\label{vert022}
(2n-1)\rho_i(\xi_i,X)+2\epsilon_i\rho_i(I_kX,\xi_j)=-\epsilon_i\frac{2n-1}{16n(n+2)}X(Scal)\\
+\frac14\Big[\epsilon_i(\nabla_{e_a}\tau)(e_a,X)-(\nabla_{e_a}\tau)(I_ie_a,I_iX)\Big]
+\epsilon_i(\nabla_{e_a}\mu)(X,e_a).
\end{multline}
Now, \eqref{vert012} and \eqref{vert022} yield \eqref{vert023} and \eqref{vert024}, which completes the proof.
\end{proof}
We  arrive to two conclusions based on Theorem~\ref{bianrrr}. First, substituting \eqref{vert024} and \eqref{vert023} into \eqref{ricciformv},
we obtain the following theorem.

\begin{thrm}{\bf The contracted second Bianchi identity.}
On a pqc manifold of dimension $4n+3$ the next formula holds
\begin{equation}  \label{div}
(n-1)(\nabla_{e_a}\tau)(e_a,X)+2(n+2)(\nabla_{e_a}\mu)(e_a,X)-\frac{(n-1)(2n+1)}{8n(n+2)}d(Scal)(X)=0.
\end{equation}
\end{thrm}
\begin{prop}\label{hvan}
Let the curvature of the canonical pqc-connection vanishes on $H, \quad  R_{_{|H}}=0$. Then, $\C$ is flat, $R=0$ and the non-zero part of the torsion is given by \eqref{hort}.
\end{prop}
\begin{proof}
The condition $R_{_{|H}}=0$ implies that all the horizontal Ricci-type tensors vanish. Then, Theorem~\ref{sixtyseven} yields $\tau=\mu=Scal=0$. These conditions and Theorem~\ref{bianrrr}, \eqref{vert023} and \eqref{vert024} lead to $\rho_s(\xi_s,X)=\rho_s(\xi_t,X)=0$, which substituted into \eqref{vert1}, \eqref{vert2} and \eqref{vertor} give
\be\label{rvert}R(\xi,X,Y,Z)=R(\xi_s,\xi_t,X,Y)=0, \qquad T(\xi_i,\xi_j,I_kX)=\f_k([\xi_i,\xi_j],X)=\rho_k(I_jX,\xi_i)=0.
\ee
In particular, the vertical distribution V is involutive. 

Taking into account $Scal=0$ and \eqref{rvert}, we get from \eqref{torv} together with \eqref{torv1} that $T(\xi_s,\xi_t)=0.$

The equation \eqref{rhov} implies
$R(X,Y,\xi_i,\xi_j)=-2\epsilon_k\rho_k(X,Y)=0, \quad R(X,\xi,\xi_i.\xi_j)=-2\epsilon_k\rho_k(X,\xi)=0.$  

\noindent A combination of \eqref{rhov} and \eqref{rvert} yields
$
2nR(\xi_s,\xi_t,\xi_i,\xi_j)=-4n\epsilon_k\rho_k(\xi_s,\xi_t)=-\epsilon_kR(\xi_s,\xi_t, e_a,I_ke_a)=0,
$ 
which ends the proof.
\end{proof}

\subsection{Local structure equations of pqc manifolds}

The fundamental 2-forms $\omega_s$ of a pqc structure are locally defined horizontal 2-forms. We define a global horizontal four form $\Omega$, whose exterior derivative contains the essential information about the torsion endomorphism of the canonical pqc-connection,   provided the dimension of the manifold is grater than seven. The   $Sp(n,\mathbb R)Sp(1,\mathbb R)$-invariant  fundamental four form of a given pqc manifold is defined globally on 
the horizontal distribution $H$   by 
\begin{equation}\label{fform}
\Omega=-\omega_1\wedge\omega_1-\omega_2\wedge\omega_2+\omega_3\wedge\omega_3.
\end{equation}
First, we derive the local structure equations  of a pqc structure in
terms of the $sp(1,\mathbb R)$-connection forms of the canonical pqc-connection and
the pqc scalar curvature.


\begin{prop}\label{str}
Let $(M^{4n+3},\eta,\mathbb{PQ})$ be a (4n+3)- dimensional pqc manifold
with pqc normalized scalar curvature $\lambda$. 
The following  equations hold
\begin{gather}\label{streq}
d\eta_i=-\epsilon_i2\omega_i+\eta_j\wedge\alpha_k+\epsilon_j\eta_k\wedge\alpha_j +\epsilon_i\lambda
\eta_j\wedge\eta_k,\end{gather}
\begin{multline}\label{str2}
\epsilon_id\omega_i=\omega_j\wedge[-\epsilon_j\alpha_k+\epsilon_k\lambda\eta_k]+\omega_k\wedge[\epsilon_i\alpha_j-\epsilon_j\lambda\eta_j]-\epsilon_j\rho_k\wedge\eta_j+
\epsilon_k\rho_j\wedge\eta_k+\frac12\epsilon_id\lambda\wedge\eta_j\wedge\eta_k,\
\end{multline}
\begin{gather}\label{strom}
d\Omega=\sum_{(ijk)}-\epsilon_i\Big[2\eta_i\wedge(\rho^0_k\wedge\omega_j-\rho^0_j\wedge\omega_k)+
d\lambda\wedge\omega_i\wedge\eta_j\wedge\eta_k\Big],
\end{gather}
where $\alpha_s$ are the $sp(1,\mathbb R)$-connection 1-forms of the canonical pqc-connection, $\sum_{(ijk)}$ is the cyclic sum of even permutations of $\{1,2,3\}$ and 
 \be\label{trr}
 \rho^0_s(X,Y)=\frac12\Big[\tau(X,I_sY)-\tau(I_sX,Y)\Big]+2\mu(X,I_sY)
 \ee
 are the trace-free part of the Ricci 2-forms.
\end{prop}

\begin{proof}
A straightforward calculation
using \eqref{aa1}, \eqref{vaa}, \eqref{convv} and \eqref{scaltor} gives the equivalence of \eqref{om1} and \eqref{streq}.
Taking the exterior derivative of \eqref{streq}, followed by an application of
\eqref{streq} and  \eqref{rhoal}  implies \eqref{str2}.   The exterior derivative of \eqref{str2} and the   definition \eqref{fform} of the 4-form $\Omega$ imply
\be\label{dfour}
d\Omega=\sum_{(ijk)}-\epsilon_i\Big[2\eta_i\wedge(\rho_k\wedge\omega_j-\rho_j\wedge\omega_k)+
d\lambda\wedge\omega_i\wedge\eta_j\wedge\eta_k\Big].
\ee
The last formula, \eqref{strom} follows from \eqref{dfour} by taking into account  \eqref{ricciformf}.
\end{proof}
The next result  expresses the tensors $\tau$ and $\mu$ in terms of the exterior derivative of the fundamental four form. We have the following
\begin{thrm}\label{lemdom}
On a pqc manifold of dimension $(4n+3)>7$ we have the identities
\begin{gather}\label{domu}
\mu(X,Y)=
-\frac1{32n}\Big[d\Omega(\xi_i,X,I_kY,e_a,I_je_a)-\epsilon_kd\Omega(\xi_i,I_iX,I_jY,e_a,I_je_a)
\Big];\\\label{domt}
\tau(X,Y)=\frac1{16(1-n)}\sum_{(ijk)}\Big[d\Omega(\xi_i,X,I_kY,e_a,I_je_a)+\epsilon_kd\Omega(\xi_i,I_iX,I_jY,e_a,I_je_a)
\Big].
\end{gather}
\end{thrm}
\begin{proof}
Equation \eqref{strom}
 yields
\be\label{rom}
d\Omega(\xi_i,X,I_kY,e_a,I_je_a)\\=-8\epsilon_k(n-1)\rho^0_k(X,I_kY)-
4\epsilon_j\rho^0_j(X,I_jY)-4\rho^0_j(I_iX,I_kY),
\ee
A substitution of \eqref{trr} in \eqref{rom}, combined with  the properties of
the tensors $\tau$ and $\mu$ described in Lemma~\ref{lts} and Lemma~\ref{lta}  give
\be\label{lasp}
d\Omega(\xi_i,X,I_kY,e_a,I_je_a)\\=-4(n-1)\Big[\tau(X,Y)-\epsilon_k\tau(I_kX,I_kY)\Big]
-16n\mu(X,Y).
\ee
{Applying again Lemma~\ref{lts} and Lemma~\ref{lta} to
\eqref{lasp}, we see that} $\mu$ and $\tau$ satisfy \eqref{domu} and
\eqref{domt}, respectively, which completes the proof.
\end{proof}

\section{The flat model}
\begin{thrm}\label{van}
Let $(M,\eta,\mathbb{PQ},g)$ be a para quaternionic contact
manifold of dimension $4n+3$. Then
$(M,\eta,\mathbb{PQ},g)$ is locally isomorphic to  the para quaternionic Heisenberg
group exactly when the canonical pqc connection has vanishing
horizontal curvature, $R(X,Y,Z,V)=0$.
\end{thrm}
\begin{proof}
First we prove the following
\begin{lemma}\label{ffl}
On a pqc manifold the torsion tensor of the canonical  pqc connection $\nabla$, restricted to $H$, is $\nabla$-parallel, $(\nabla_AT)(X,Y)=0.$
\end{lemma}
The equality \eqref{nablat} yields $(\nabla_AT)(X,Y,Z)=0$. 

Since $\nabla$ preserves the orthogonal splitting $H\oplus V$, then \eqref{hort}, \eqref{torv1}, \eqref{torvv}, \eqref{exv} and \eqref{exh}  yield
\begin{multline}
(\C_AT)(Y,Z,\xi_i)=2(\C_A\f_i)(Y,Z)-2\sum_{s=1}^3\epsilon_s\f_s(Y,Z)g(\xi_s,\C_A\xi_i)\\=2\alpha_j(A)\f_k(Y,Z)+2\epsilon_k\alpha_k(A)\f_j(Y,Z)
-2\alpha_j(A)\f_k(Y,Z)-2\epsilon_k\alpha_k(A)\f_j(Y,Z)=0
\end{multline}
which proves  the Lemma~\ref{ffl}.  

It is easy to see that the canonical pqc connection on the para quaternionic
Heisenberg group is the left-invariant connection on the group
which is flat and the torsion is non-vanishing only on $H, \quad T=T_{_{|H}}$. 

For the converse, by applying Proposition~\ref{hvan} we can conclude that $\C$ is flat and the torsion is non-zero only on $H$. Taking into account Lemma~\ref{ffl}, we conclude that the torsion is parallel, $\C T=0$ and the first Bianchi identity \eqref{bian1} reads
\be\label{lieal}
T(T(A,B),C)+T(T(B,C),A)+T(T(C,A),B)=0
\ee
Hence, the manifold has a local Lie group structure $T$ by the Lie theorem. The structure equations of this Lie group determined by  \eqref{hort} are
$d\eta_s=-2\epsilon_s\f_s$ which are precisely the structure equations of the para-quaternionic
Heisenberg group. Therefore, by applying again the Lie theorem, we can conclude that
the manifold has a local Lie group structure, which is locally isomorphic to
$G(pH)$. In other words, there is a local diffeomorphism $\Phi : M\rightarrow G(pH)$, such that $\eta=\Phi*\tilde\Theta$, where $\tilde\Theta$ is the 
the standard contact form on $G(pH)$, see \eqref{pqh}.
\end{proof}

\section{pqc-Einstein paraquaternionic contact structures}\label{ss:qc-Einstein}

The aim of this section is to show that the vanishing of the
torsion endomorphism of the canonical  pqc connection implies that the
pqc-scalar curvature is constant. The Bianchi identities will have
an important role in the analysis.
\begin{dfn}\label{d:qc Einstein}
A pqc structure is \emph{pqc-Einstein} if the
pqc-Ricci tensor is trace-free, 
\be\label{pqcEin}
Ric(X,Y)=\frac{Scal}{4n}g(X,Y)
\ee
\end{dfn}

The next result describes the structure of the pqc-Einstein spaces.
\begin{thrm}\label{tor-ein}
Let $(M,g,\mathbb{PQ})$ be a para-quaternionic contact manifold of dimension $(4n+3)$. Then,
\begin{itemize}
\item[a).] $(M,g,\mathbb{PQ})$ is a pqc-Einstein if and only if the
tensors $\tau=\mu=0$, i.e. the torsion endomorphism vanishes
identically, $T(\xi,X)=0$. \item[b).] On a pqc-Einstein
manifold of dimension bigger than seven the pqc scalar curvature is
constant, $d(Scal)=0$ and the vertical space spanned by the
Reeb vector fields is integrble, $[\xi_s,\xi_t]\in V.$ \item[c).]
If $n>1$, then $(M,g,\mathbb{PQ})$ is  pqc-Einstein if and only if   the
fundamental four form is closed, $d\Omega=0.$
\end{itemize}
\end{thrm}
\begin{proof}
Part a) of the assertion follows from \eqref{ricci} and the defining condition \eqref{pqcEin}.

The first part of b)  is a consequence of part a) and \eqref{div}, since $n>1$. Substitute $T^0=U=d(Scal)=0$ into \eqref{vert024} to conclude $\rho_i(\xi_j,X)=0$ and compare this with \eqref{vertor} to establish the integrability of the vertical distribution $V$.

To proof c), assume $T^0=U=0$ and $n>1$. Then,
Theorem~\ref{sixtyseven} implies
\begin{equation}\label{e:rhota for Omega closed}
\begin{aligned}
\rho_s(X,Y)=-\lambda\omega_s(X,Y), \quad \rho_s(\xi_t,X)=0, \\
\epsilon_i\rho_i(\xi_i,\xi_j)+\epsilon_j\rho_k(\xi_k,\xi_j)=0,
\end{aligned}
\end{equation}
since $Scal$ is a constant and the horizontal distribution is
integrable. Using the just obtained identities in \eqref{e:rhota for Omega closed}, we derive from \eqref{strom}
that $d\Omega=0$.

The converse of c) follows directly from Theorem~\ref{lemdom}, which completes the proof of the theorem.
\end{proof}
The well known Cartan formula applied  for the fundamental four form gives
$$\mathbb L_{\xi_s}\Omega=\xi_s\lrcorner
d\Omega+d(\xi_s\lrcorner\Omega)=\xi_s\lrcorner d\Omega,
$$
since $\Omega$ is horizontal. The latter formula and
Theorem~\ref{lemdom} together with Theorem~\ref{tor-ein} yield
\begin{cor}\label{lemdomu}
If  one of the Reeb vector fields preserves the fundamental four
form on a pqc manifold of dimension $(4n+3)>7$, then $\mu=0$ and the
torsion endomorphism of the canonical connection is symmetric, $T_{\xi_s}=\tb$.

If on a pqc manifold of dimension $(4n+3)>7$ each Reeb vector field preserves the fundamental four form, $\mathbb L_{\xi_s}\Omega=0$,   then  the
torsion endomorphism of the canonical connection vanishes, $T_{\xi_s}=\tau=\mu=0$ and the manifold is pqc-Einstein.
\end{cor}

Basic examples of pqc-Einstein spaces are provided by the para 3-Sasakian spaces. Indeed,  in view of \eqref{ricci} the pqc-Einstein condition is equivalent to the fact  that the torsion endomorphism vanishes, $\tau=\mu=0$ and Proposition~\ref{p3sas} implies that any para 3-Sasakian space is pqc-Einstein. More precisely, we have
\begin{prop}
Any para 3-Sasakian manifold is a pqc-Einstein with pqc-scalar
curvature 
\be\label{scal3sas}Scal = 16n(n + 2).
\ee
The structure equations of a para 3-Sasakian manifolds are the equations \eqref{co}.

The pqc Ricci-type tensors of para 3-Sasakian manifolds are given by
\be\label{ric3sas}
\begin{split}
\rho_s(X,Y)=\varrho_s(X,Y)=-2\zeta_s(X,Y)=-2\f_s(X,Y);\\
Ric(\xi_s,X)=\rho_s(\xi_t,X)=\zeta_s(\xi_t,X)=\rho_s(\xi_t,\xi_r)=0
\end{split}
\ee
The curvature $R$ of the canonical pqc connection is expressed in terms of the
curvature of the Levi-Civita connection $R^g$ as follows
\begin{gather}\label{rrg1}
R(X,Y,Z,V)=R^g(X,Y,Z,V)\\\nonumber+\sum_{s=1}^3\Big[\epsilon_s\f_s(X,Z)\f_s(Y,V)-\epsilon_s\f_s(Y,Z)\f_s(X,V)-2\f_s(X,Y)\f_s(Z,V)\Big];\\\label{rrg2}
0=R(\xi_s,Y,Z,V)=R^g(\xi_s,Y,Z,V),\\\label{rrg3} 0= R(\xi_s,\xi_t,Z,V)=R^g(\xi_s,\xi_t,Z,V)=R^g(Z,V,\xi_s,\xi_t).
\end{gather}
\end{prop}
\begin{proof}
For a para 3-Sasakian manifolds, the equalities \eqref{co}, \eqref{conv1}, \eqref{convv} and \eqref{scaltor} imply
\be\label{eta3sas}
2\alpha_i=-\epsilon_j(2+\lambda)\eta_i
\ee
We calculate from \eqref{rhoal} using \eqref{eta3sas} and \eqref{co} that
$\rho_i(X,Y)=\frac12\epsilon_kd\alpha_i(X,Y)=-(1+\frac{\lambda}2)\f_i(X,Y),
$
which compared with the first equality in \eqref{e:rhota for Omega closed} gives $\lambda=2$ which combined with \eqref{scal} proves \eqref{scal3sas}. The equalities \eqref{ric3sas},\eqref{rrg2} and \eqref{rrg3} follow from \eqref{clcon}, Theorem~\ref{sixtyseven} and  Theorem~\ref{bianrrr} taking into account $\tau=\mu=0, \lambda=2$ and the properties of the curvature of the Levi-Civita connection.

The equalities \eqref{clcon} and the fact that the vertical space is integrable  yield
\be\label{clc}
\C_YZ=\C^g_YZ-\sum_{s=1}^3\epsilon_s\f_s(Y,Z)\xi_s,\quad \C_X\xi_i=\C^g_X\xi_i+I_iX.
\ee
The first equality in \eqref{clc} 
implies \eqref{rrg1}.
\end{proof}
It turns out that the para 3-Sasakian spaces are locally the only
pqc-Einstein manifolds. We have

\begin{thrm}\label{main2}
Let $(M^{4n+3},\eta,pQ)$ be a $4n+3$-dimensional pqc
manifold with non-zero pqc  scalar curvature $Scal$. For $n>1$ the
following conditions are equivalent
\begin{enumerate}
\item[a)] {$(M^{4n+3},g,pQ)$ is pqc-Einstein manifold};
\item[b)] $M^{4n+3}$ is locally pqc homothetic to a para 3-Sasakian manifold, i.e., locally, there exists a
$SO(1,2)$-matrix $\Psi$ with smooth entries depending on an
{auxiliary} parameter, such that the local pqc structure
$(\frac{16n(n+2)}{Scal}\Psi\cdot\eta,pQ)$ is para 3-Sasakian.
\end{enumerate}
\end{thrm}
\begin{proof}
Let $\tau=\mu=0$ and $n>1$. Theorem~\ref{tor-ein} {shows}
that the pqc scalar curvature is constant and the vertical
distribution is integrable. The pqc structure
$\eta'=\frac{16n(n+2)}{\epsilon Scal}\eta$ has normalized pqc scalar curvature
$\lambda'=2$ and $d\Omega'=0$, provided $Scal\not=0$. For simplicity, we
shall denote $\eta'$ with $\eta$ {and, in fact, omit the $'$
everywhere}.

In the first step of the proof we show that the pseudo Riemannian cone $N=M\times\mathbb R^+$ with the
metric $g_N=t^2(g-\sum_{s=1}^3\epsilon_s\eta_s\otimes\eta_s) + dt\otimes dt$
has holonomy contained  in $Sp(n+1,\mathbb R)$, i.e. it is hypersymplectic. To this end we consider the
following four form on $N$
\begin{equation}\label{con4}
    F=-\epsilon_iF_i\wedge F_i-\epsilon_jF_j\wedge F_j-\epsilon_kF_k\wedge F_k,
\end{equation}
where the two forms $F_s$ are defined by
\begin{equation}\label{con2}
F_i=t^2(\omega_i+\eta_j\wedge\eta_k)+\epsilon_it\eta_i\wedge dt
\end{equation}

Applying \eqref{streq}, \eqref{str2} and
\eqref{strom}, we calculate from \eqref{con2} applying \eqref{e:rhota for Omega closed} and $\lambda=2$ that

\begin{multline}\label{nnnn}
dF_i=tdt\wedge\Bigl(2\omega_i+2\eta_j\wedge\eta_k+\epsilon_i
d\eta_i\Bigr)+t^2d(\omega_i+\eta_j\wedge\eta_k)\\
=t\, dt\wedge \Bigl(4\eta_j\wedge\eta_k+\epsilon_i\eta_j\wedge\alpha_k-\epsilon_k\eta_k\wedge\alpha_j \Bigr)\\
+t^2\Big[\omega_j\wedge(\epsilon_k\alpha_k-\epsilon_js\eta_k)+\omega_k\wedge(\alpha_j+\epsilon_ks\eta_j)+\epsilon_k\rho_k\wedge\eta_j-
\epsilon_j\rho_j\wedge\eta_k\Big]
\\
- t^2 \Bigl(2\epsilon_j\omega_j-\epsilon_k\eta_i\wedge\alpha_k\Bigr)\wedge\eta_k+ t^2 \Bigl(2\epsilon_k\omega_k-\eta_i\wedge\alpha_j\Bigr)\wedge\eta_j\\
=t\, dt\wedge \Bigl(4\eta_j\wedge\eta_k+\epsilon_i\eta_j\wedge\alpha_k-\epsilon_k\eta_k\wedge\alpha_j \Bigr)\\
+t^2\Big[\epsilon_k\omega_j\wedge\alpha_k+\omega_k\wedge\alpha_j\Big]
- t^2 \Bigl(2\epsilon_j\omega_j-\epsilon_k\eta_i\wedge\alpha_k\Bigr)\wedge\eta_k+ t^2 \Bigl(2\epsilon_k\omega_k-\eta_i\wedge\alpha_j\Bigr)\wedge\eta_j
\end{multline}
 A short computation, using \eqref{streq}, \eqref{str2},
\eqref{strom} and \eqref{nnnn}, gives
\begin{multline}
\frac12dF=-\sum_{s=1}^3\epsilon_s dF_s\wedge F_s
=t^3dt\wedge\sum_{(ijk)} \Bigl[-4\epsilon_i\omega_i\wedge\eta_k\wedge\eta_j+2\epsilon_j\f_j\wedge\eta_k\wedge\eta_i-2\epsilon_k\f_k\wedge\eta_j\wedge\eta_i \Bigr]\\
 -t^3dt\wedge\sum_{(ijk)} \left [\f_i\wedge\eta_j\wedge\alpha_k+\epsilon_j\f_i\wedge\eta_k\wedge\alpha_j+\epsilon_k\f_j\wedge\alpha_k\wedge\eta_i+\f_k\wedge\alpha_j\wedge\eta_i\right]\\
 + t^4\sum_{(ijk)} \left [\epsilon_j\f_i\wedge\f_j\wedge\alpha_k-\epsilon_i\f_i\wedge\f_k\wedge\alpha_j\right]+t^4\sum_{(ijk)}\left[2\epsilon_j\f_i\wedge\f_k\wedge\eta_j-2\epsilon_k\f_i\wedge\f_j\wedge\eta_k\right ] \\
  + t^4\sum_{(ijk)} \left [\epsilon_j\f_j\wedge\eta_j\wedge\eta_k\wedge\alpha_k-\epsilon_i\f_k\wedge\eta_j\wedge\eta_k\wedge\alpha_j-\epsilon_j\f_i\wedge\eta_i\wedge\eta_k\wedge\alpha_k-\epsilon_i\f_i\wedge\eta_i\wedge\eta_j\wedge\alpha_j\right]=0.
\end{multline}
Hence, $dF=0$ and the holonomy of the cone metric is contained in $Sp(n+1,\mathbb R)Sp(1,\mathbb R)$,  provided $n>1$
\cite{Swann},  i.e. the cone is para-quaternionic K\"ahler manifold,
provided $n>1$. 

It is well known (see e.g \cite{Swann}) that a para-quaternionic K\"ahler manifolds of dimension
bigger than four  are Einstein. This fact
implies that the cone $N=M\times \mathbb R^+$ with the  metric
$g_N$ must be Ricci flat  and
therefore it is locally hyper-para-k\"ahler, { since the $sp(1,\mathbb R)$-part
of the Riemannian curvature vanishes and therefore it can be
trivialized locally by a parallel sections} (see e.g. \cite{Swann}). This means that locally there exists a
$SO(1,2)$-matrix $\Psi$ with smooth entries, possibly depending on
$t$,  such that the triple of two forms $(\tilde F_1,\tilde
F_2,\tilde F_3)= \Psi\cdot (F_1,F_2,F_3)^T$ consists of closed
2-forms defining a local
hyper-para-k\"ahler structure. Consequently, $(M,\Psi\cdot\eta)$ is
locally a para 3-Sasakian manifold \cite{Swann}.

The fact that b) implies a) is trivial in view of Theorem~\ref{tor-ein} since the 4-form $\Omega$ is invariant under hyperbolic rotations and rescales by a constant when the metric on the horizontal space $H$ is replaced by another metric, homothetic to it.
\end{proof}

\begin{rmrk}
An example of a pqc structure satisfying $\tau=\mu=Scal=0$ can be obtained as follows.
Let $M^{4n}$ be a hyper-para-k\"ahler (hypersymplectic) 
manifold with closed and locally exact K\"ahler forms
$\omega_l=d\eta_l$. The total space of an ${\mathbb R^3}$-bundle over the hyper-para-k\"ahler
manifold $M^{4n}$ with connection 1-forms $\eta_l$ is an example of a pqc structure with $\tau=\mu=Scal=0$. The pqc
structure is determined by the three 1-forms $\eta_l$ satisfying
$d\eta_l=\omega_l$, which yield $\tau=\mu=Scal=0$. In particular, the
para quaternionic Heisenberg group, which locally is the unique pqc
structure with flat canonical connection,  can
be considered as an $\mathbb R^3$ bundle over a 4n-dimensional
flat hyper-para-k\"ahler $\mathbb R^{4n}$. A compact example is provided
by a $T^3$-bundle over a compact hyperk-para-k\"ahler manifold $M^{4n}$,
such that each closed K\"ahler form $\omega_l$ represents integral
cohomology classes. Indeed, since $[\omega_l]$, $1\leq l\leq 3$
defines integral cohomology classes on $M^{4n}$, the well-known
result of Kobayashi \cite{Kob} implies that there exists a circle
bundle $S^1 \hookrightarrow M^{4n+1} \to M^{4n}$ with connection
$1$-form $\eta_1$ on $M^{4n+1}$, whose curvature form is $d\eta_1 =
\omega_1$. Because $\omega_l$ $(l=2,3)$ defines an integral
cohomology class on $M^{4n+1}$, there exists a principal circle
bundle $S^1 \hookrightarrow M^{4n+2} \to M^{4n+1}$ corresponding
to [$\omega_2$] and a connection $1$-form $\eta_2$ on $M^{4n+2}$,
such that $\omega_2=d\eta_2$ is the curvature form of $\eta_2$.
Using again the result of Kobayashi, one gets a $T^3$-bundle over
$M^{4n}$, whose total space has a pqc structure  satisfying
$d\eta_l=\omega_l$, which yield $\tau=\mu=Scal=0$.
\end{rmrk}

\end{document}